\documentclass[12pt,a4paper,twoside,reqno]{amsart}
\allowdisplaybreaks

\usepackage[all,cmtip]{xy}
\usepackage[T1]{fontenc}

\usepackage{amsmath,amscd}
\usepackage{amssymb,amsfonts}
\usepackage[driver=pdftex,margin=3cm,heightrounded=true,centering]{geometry}
\usepackage{mathtools}
\usepackage{tensor}
\usepackage{url}
\usepackage[colorlinks=true,linkcolor=blue,citecolor=blue]{hyperref}
\usepackage{slashed}
\usepackage{mathscinet}

\usepackage{graphicx}

\usepackage{thmtools}

\usepackage{amsthm}

\theoremstyle{plain}

\newtheorem{theorem}{Theorem}[section]
\newtheorem*{thm*}{Theorem}
\newtheorem{prop}[theorem]{Proposition}
\newtheorem{lemma}[theorem]{Lemma}
\newtheorem{cor}[theorem]{Corollary}

\theoremstyle{definition}
\newtheorem{dfn}[theorem]{Definition}

\theoremstyle{remark} 

\theoremstyle{plain}

\usepackage{amscd}

\numberwithin{equation}{section}

\newcommand{\alpheqn}[1][\relax]{
     \refstepcounter{equation}
     \if#1\relax \relax
       \else \label{#1}
     \fi  
     \setcounter{saveeqn}{\value{equation}}%
    \setcounter{equation}{0}%
    \renewcommand{\theequation}{\thealphequation}}
\newcommand{\reseteqn}{\setcounter{equation}{\value{saveeqn}}%
     \renewcommand{\theequation}{\thearabicequation}}

\providecommand{\mathscr}{\mathcal} 
\IfFileExists{mathrsfs.sty}{
\usepackage{mathrsfs}}         
   {
     \IfFileExists{eucal.sty}{
        \usepackage[mathscr]{eucal}} 
   {
   }
}

\newcommand{\Lip}{\operatorname{Lip}}
\newcommand{\mk}{\operatorname{mk}}
\newcommand{\Sp}{\operatorname{Sp}}

\newcommand{\vertiii}[1]{{\left\vert\kern-0.25ex\left\vert\kern-0.25ex\left\vert #1 
    \right\vert\kern-0.25ex\right\vert\kern-0.25ex\right\vert}}
\newcommand{\Bvert}[1]{{\Big\vert\kern-0.25ex\Big\vert\kern-0.25ex\Big\vert #1 
    \Big\vert\kern-0.25ex\Big\vert\kern-0.25ex\Big\vert}}
\newcommand{\bvert}[1]{{\big\vert\kern-0.25ex\big\vert\kern-0.25ex\big\vert #1 
    \big\vert\kern-0.25ex\big\vert\kern-0.25ex\big\vert}}
\newcommand{\nvert}[1]{{\vert\kern-0.25ex\vert\kern-0.25ex\vert #1 
    \vert\kern-0.25ex\vert\kern-0.25ex\vert}}

\newcommand{\dirac}{\slashed D}

\renewcommand{\leq}{\leqslant}
\renewcommand{\geq}{\geqslant}

\newcommand{\cd}{\cdot}
\newcommand{\clc}{\cdot\ldots\cdot}
\newcommand{\ot}{\otimes}
\newcommand{\hot}{\widehat \otimes}

\newcommand{\op}{\oplus}

\newcommand{\ci}{\circ}

\newcommand{\ti}{\times}
\newcommand{\nn}{\mathbb{N}}
\newcommand{\zz}{\mathbb{Z}}

\newcommand{\cc}{\mathbb{C}}

\newcommand{\al}{\alpha}
\newcommand{\be}{\beta}
\newcommand{\ga}{\gamma}
\newcommand{\Ga}{\Gamma}
\newcommand{\de}{\delta}
\newcommand{\De}{\Delta}
\newcommand{\ep}{\varepsilon}

\newcommand{\io}{\iota}

\newcommand{\la}{\lambda}
\newcommand{\La}{\Lambda}
\newcommand{\Na}{\nabla}
\newcommand{\om}{\omega}
\newcommand{\Om}{\Omega}
\newcommand{\si}{\sigma}

\newcommand{\te}{\theta}

\newcommand{\ze}{\zeta}
\newcommand{\pa}{\partial}
\newcommand{\da}{\dagger}

\newcommand{\ov}{\overline}
\newcommand{\C}[1]{\mathcal{#1}}
\newcommand{\G}[1]{\mathfrak{#1}}
\newcommand{\T}[1]{\textup{#1}}

\newcommand{\B}[1]{\mathbb{#1}}

\newcommand{\fork}[2]{\left\{ \begin{array}{#1} #2 \end{array} \right.}

\newcommand{\ma}[2]{\left(\begin{array}{#1} #2 \end{array} \right)}

\newcommand{\su}{\subseteq}

\newcommand{\wit}{\widetilde}

\newcommand{\inn}[1]{\langle #1 \rangle}

\newcommand{\sem}{\setminus}

\makeatletter
\@namedef{subjclassname@2020}{%
  \textup{2020} Mathematics Subject Classification}
\makeatother

\begin{document}

\title[Spectral metrics on quantum projective spaces]{Spectral metrics on quantum projective spaces}


\author{Max Holst Mikkelsen}
\address{Department of Mathematics and Computer Science,
The University of Southern Denmark,
Campusvej 55, DK-5230 Odense M,
Denmark}
\email{maxmi@imada.sdu.dk}

\author{Jens Kaad}
\address{Department of Mathematics and Computer Science,
The University of Southern Denmark,
Campusvej 55, DK-5230 Odense M,
Denmark}
\email{kaad@imada.sdu.dk}

\subjclass[2020]{58B34; 58B32, 46L30} 

\keywords{Quantum metric spaces, Quantum projective spaces, Spectral triples, Quantum $SU(N)$.} 



\begin{abstract}
  We show that the noncommutative differential geometry of quantum projective spaces is compatible with Rieffel's theory of compact quantum metric spaces. This amounts to a detailed investigation of the Connes metric coming from the unital spectral triple introduced by D'Andrea and D\k{a}browski. In particular, we establish that the Connes metric metrizes the weak-$*$ topology on the state space of quantum projective space. This generalizes previous work by the second author and Aguilar regarding spectral metrics on the standard Podle\'s spheres. 
\end{abstract}

\maketitle
\tableofcontents

\section{Introduction}
In the last two decades it has become more and more apparent that one of the key open problems in noncommutative geometry is to reconcile this theory with the theory of compact quantum groups, \cite{CoMo:TIII,Con:NCG,Wor:CQG}. An important aspect of this problem concerns the construction of spectral triples on $q$-deformations of classical Lie groups in a way which is compatible with the $q$-geometry as witnessed by the quantized enveloping algebra, \cite{Dri:QG,Jim:QYB}. We note that there are currently interesting examples of spectral triples available for $q$-deformations as described in \cite{DLSSV:DOS,CaPa:EST} for the case of quantum $SU(2)$ and in the general setting in \cite{NeTu:DOC}. These examples are however less sensible to the underlying $q$-geometry since the spectra of the corresponding abstract Dirac operators are similar to the spectra of their classical counter parts. Studies of the important example quantum $SU(2)$ do in fact suggest that the current spectral triple framework for noncommutative geometry (with or without twists) is too inflexible to capture all of the rich geometric features of this compact quantum group, \cite{KRS:RFH,KaSe:TST,KaKy:SU2,Bra:GFC}.

Instead, it appears that some of the quantum flag manifolds are more amenable to $q$-geometric investigations using spectral triples as the main theoretical framework, see \cite{Kra:DOQ,KrTu:DDQ} where the relationship to the covariant first order differential calculus found by Heckenberger and Kolb is also highlighted, \cite{HeKo:LDI,HeKo:RQF}. The hope is then that these quantum flag manifolds can be used as building blocks for understanding larger and larger parts of the $q$-deformed classical Lie groups through the lens of noncommutative geometry, see e.g. \cite{KRS:RFH,KaKy:SU2} for applications of this idea in the case of quantum $SU(2)$. The correct way of assembling these building blocks is in principle dictated by unbounded $KK$-theory as pioneered by Mesland in \cite{Mes:UCN} and consolidated in \cite{KaLe:SFU,MeRe:NMU}. It should however be noted that the current version of unbounded $KK$-theory features a too restrictive focus on unbounded Kasparov modules (bivariant spectral triples), meaning that the incorporation of twisted commutators is currently not allowed in unbounded $KK$-theory, even though some progress in this direction has been made in \cite{Kaa:UKM}. 

We are in this text studying the noncommutative geometry of the higher dimensional quantum projective spaces which are quantum flag manifolds associated with quantum $SU(N)$. The point of departure is the unital spectral triples introduced by D'Andrea and D\k{a}browski in \cite{DaDa:DQP} building on the earlier work \cite{DaDaLa:NGQPP} including Landi as a coauthor.

One of the many fundamental insights of Connes was that any unital spectral triple gives rise to an (extended) metric on the state space of the associated unital $C^*$-algebra, \cite{Con:CFH}. It was then emphasized by Rieffel that it is important to ask whether the Connes metric metrizes the weak-$*$ topology or not. This question leads to the whole theory of compact quantum metric spaces as pioneered by Rieffel in \cite{Rie:MSA,Rie:MSS} and prominently featuring a quantum analogue of the Gromov-Hausdorff distance between compact metric spaces, \cite{Rie:GHD}. The quantum Gromov-Hausdorff distance has later on been fine tuned to incorporate more and more structure in the deep and extensive work of Latr\'emoli\`ere, \cite{Lat:QGH,Lat:MGH,Lat:GHM}.

The main result of this paper can be formulated as follows:
\medskip

\emph{The Connes metric metrizes the weak-$*$ topology on the state space of quantum projective space. In other words, the unital spectral triple introduced by D'Andrea and D\k{a}browski turns quantum projective space into a spectral metric space.}
\medskip

As a special case of this theorem we recover the main result from \cite{AgKa:PSM} which states that the standard Podle\'s sphere becomes a spectral metric space when equipped with the D\k{a}browski-Sitarz spectral triple introduced in \cite{DaSi:DSP}.

Before giving some details on our method of proof we pay attention to a certain detail regarding the coordinate algebra appearing in a spectral triple. This choice of coordinate algebra influences the definition of the Connes metric on the state space and the larger this algebra is, the harder it becomes to recover the weak-$*$ topology. We are in this text working with the \emph{Lipschitz algebra} which is the maximal choice of coordinate algebra (with respect to inclusion) and our result regarding spectral metrics on quantum projective space can therefore not be strengthened. A part from maximality, another feature of the Lipschitz algebra is that it can be reconstructed from the $C^*$-algebra and the abstract Dirac operator of the spectral triple (making it canonical in this sense as well). 
\medskip

In the proof of our main theorem we rely on the structure of quantum $SU(N)$ as a $C^*$-algebraic compact quantum group. More precisely, we apply the coaction of quantum $SU(N)$ on quantum projective space together with a sequence of states converging in the weak-$*$ topology to the counit (restricted to quantum projective space). Each of the states appearing in this sequence provides us with a finite dimensional approximation of the inclusion of quantum projective space into quantum $SU(N)$. It turns out that the precision of this approximation can be estimated by the distance in the Connes metric between the state in question and the restricted counit. 

Applying the description of compact quantum metric spaces in terms of finite dimensional approximations appearing in \cite{Kaa:ExPr}, we see that our main result follows, if we can establish that the sequence of states converges to the restricted counit with respect to the Connes metric (and not just in weak-$*$ topology). Establishing this convergence result is still a difficult task, but it can be made substantially less complicated by considering the invariance properties of the Connes metric. We show that the Connes distance between our states is invariant under a conditional expectation from quantum projective space onto a much smaller and commutative unital $C^*$-subalgebra, which is $*$-isomorphic to the continuous functions on the subset $I_q = \{ q^{2m} \mid m \in \nn_0\} \cup \{0\}$ of the real line.

We refer to $I_q$ as the \textit{quantized interval} and the above observations ensure that we only need to understand the spectral metric properties of $I_q$ (equipped with the unital spectral triple obtained by restriction of the unital spectral triple on quantum projective space). This is a manageable task and it turns out that the relevant metric information can be extracted from difference quotients for continuous functions on $I_q$ and the single positive continuous function $q^{-1} x(1 - q^2 x)$. Notably, this latter function (without the factor $q^{-1}$) also plays a major role for the continuity properties of the Podle\'s spheres in quantum Gromov-Hausdorff distance investigated in \cite{GKK:QI,AgKaKy:PSC}.

\subsubsection*{Acknowledgements}
The authors gratefully acknowledge the financial support from the Independent Research Fund Denmark through grant no.~9040-00107B, 7014-00145B and 1026-00371B. This research is part of the EU Staff Exchange project 101086394 "Operator Algebras That One Can See". 

We would like to thank Marc Rieffel and Walter van Suijlekom for making us aware of the articles \cite{Rie:MBM} and \cite{DLSSV:DOS}. It is also a pleasure to thank our colleague David Kyed for valuable input to our proof of the extension theorem.

\section{Preliminaries}
All Hilbert spaces in this text are assumed to be separable and for such a Hilbert space $H$ we let $\B B(H)$ denote the unital $C^*$-algebra of bounded operators on $H$. The inner product on the Hilbert space $H$ (which is anti-linear in the first variable) is denoted by $\inn{\cd,\cd}$ and the notation $\| \cd \|_2$ refers to the corresponding norm on $H$. The unique $C^*$-norm on a $C^*$-algebra is written as $\| \cd \|$.


Throughout this article, the deformation parameter $q$ is fixed and belongs to the open unit interval $(0,1)$. We shall moreover fix a natural number $\ell \in \nn$ and put $N := \ell + 1$.

For integers $i,j \in \zz$ we apply the notation $\de_{ij}$ for the corresponding Kronecker delta and for a natural number $n \in \nn$ we put $\inn{n} := \{1,\ldots,n\}$ and $\inn{0} := \emptyset$. The unital $C^*$-algebra of matrices of size $n$ is denoted by $\B M_n(\cc)$ and the corresponding matrix units are written as $e_{ij}$ for $i,j$ in the index set $\inn{n}$.   

\subsection{Compact quantum metric spaces}\label{introCQMS}

In this subsection, we present Rieffel's notion of a compact quantum metric space, see \cite{Rie:MSA, Rie:MSS, Rie:GHD}. We notice that the current trend is to describe the theory in the generality of operator systems in line with recent developments in noncommutative geometry, \cite{CoSu:STN,CoSu:TRO,Rie:CFT}. Keeping in mind the scope of the present paper, we are going to focus on the $C^*$-algebraic framework for compact quantum metric spaces. 

Let $\C A$ be a norm-dense unital $*$-subalgebra of a unital $C^*$-algebra $A$ and consider a seminorm $L : \C A \to [0,\infty)$. The following terminology can be found in \cite{Rie:MBM}:

\begin{dfn}
We say that $L : \C A \to [0,\infty)$ is a \textit{slip-norm} if $L(x)=L(x^*)$ for all $x\in \C A$ and $L(1) = 0$.
\end{dfn}

From now on, we shall assume that $L : \C A \to [0,\infty)$ is a slip-norm. We define an extended metric -- referred to as the \textit{Monge-Kantorovich} metric -- on the state space of $A$ by putting
\[
\mk_L(\mu,\nu):= \sup\big\{ |\mu(x)-\nu(x)| \mid x\in \C A \, ,  \, \, L(x) \leq 1  \big\}  ,
\]
for all $\mu,\nu \in S(A)$. The adjective ``extended'' means that $\mk_L$ can take the value infinity but otherwise satisfies the usual requirements of a metric.

\begin{dfn}
 We say that the pair $(\C A,L)$ is a \textit{compact quantum metric space} if $\mk_L$ metrizes the weak-$*$ topology on the state space $S(A)$. 
\end{dfn}

If $(\C A,L)$ is a compact quantum metric space, then it follows from connectedness of $S(A)$ in the weak-$*$ topology that the Monge-Kantorovich metric is actually finite so that $(S(A),\mk_L)$ is a metric space in the usual sense. We now quote a result which can be found in \cite[Proposition 1.6]{Rie:MSA} and \cite[Proposition 2.2]{Rie:MSS}. The notation $[ \cd ] : A \to A / \B C 1$ refers to the quotient map and the quotient space $A/ \B C 1$ is equipped with the quotient norm $\| \cd \|_{A/ \B C 1}$ (coming from the $C^*$-norm on $A$).

\begin{prop}\label{p:diam}
The following are equivalent: 
\begin{enumerate} 
\item There exists a constant $C > 0$ such that  $\big\| [x] \big\|_{A/\B C 1} \leq C \cdot L(x)$ for all $x\in \C A$; 
\item $\, \sup\big\{ \mk_L(\mu,\nu) \mid \mu,\nu \in S(A)\big\}< \infty$.  
\end{enumerate}
\end{prop} 

We say that the pair $(\C A,L)$ has \textit{finite diameter}, if the equivalent conditions from Proposition \ref{p:diam} are satisfied. Notice that if $(\C A,L)$ has finite diameter, then the kernel of $L$ agrees with $\B C 1$. Recently, it was shown in \cite{KyNe:FMS} that the pair $(\C A,L)$ has finite diameter if and only if the Monge-Kantorovich metric is finite (meaning that the extended metric $\mk_L$ is a metric). As a consequence, we see that every compact quantum metric space has finite diameter. This latter fact can also be deduced from compactness of the state space in the weak-$*$ topology. 

We shall now introduce a convenient characterization of compact quantum metric spaces which was obtained in \cite[Theorem 3.1]{Kaa:ExPr}. The proof relies to some extent on Rieffel's characterization of compact quantum metric spaces, see \cite[Theorem 1.8]{Rie:MSA}. We first need a definition. 

\begin{dfn}
Fix a constant $\ep >0$, a unital $C^*$-algebra $B$ and two unital linear maps $\iota,\Phi : A \to B$. Suppose that $\iota$ is isometric and that $\Phi$ is positive. We say that the pair $(\iota,\Phi)$ is an \textit{$\ep$-approximation} of $(\C A,L)$ if
\begin{enumerate}
\item The image of $\Phi$ is finite dimensional (as a vector space over $\cc$); 
\item $\| \iota(x) -\Phi(x) \| \leq \ep \cdot L(x)$ for all $x\in \C A$. 
\end{enumerate}
\end{dfn}

\begin{theorem}\label{t:approx}
The following are equivalent: 
\begin{enumerate}
\item $(\C A,L)$ is a compact quantum metric space. 
\item $(\C A,L)$ has finite diameter and for every $\ep >0$ there exists an $\ep$-approximation of $(\C A,L)$. 
\end{enumerate}
\end{theorem}

As a first example, consider a classical compact metric space $(X,d)$ and let $\Lip(X)$ denote the Lipschitz functions on $X$, sitting as a norm-dense unital $*$-subalgebra of the unital $C^*$-algebra $C(X)$ of continuous functions on $X$. We equip the Lipschitz functions with the slip-norm $L_d$ defined by 
\[
L_d(f) := \sup\Big\{ \frac{| f(p_1)-f(p_2) |}{d(p_1,p_2) }\mid p_1\neq p_2  \Big\} ,
\]
so that $L_d(f)$ agrees with the Lipschitz constant of a Lipschitz function $f$ (the above supremum is by convention equal to zero in the case where $X$ is a singleton). It then holds that the pair $\big( \Lip(X), L_d\big)$ is a compact quantum metric space, see for example \cite[Theorem 2.5.4]{RaRu:MTP}. For a more recent proof using $\ep$-approximations we refer to \cite[Lemma 3.4]{Kaa:ExPr}.

An important class of noncommutative examples arises from unital spectral triples, see \cite{Con:CFH} and \cite[Chapter 6]{Con:NCG}. Indeed, suppose that we have a unital spectral triple $(\C A,H,D)$ and let us identify $\C A$ with a unital $*$-subalgebra of the bounded operators on $H$ (via a fixed unital injective $*$-homomorphism). We define the unital $C^*$-algebra $A \su \B B(H)$ as the closure of $\C A$ in operator norm. The unital spectral triple gives rise to a closable derivation $d : \C A \to \B B(H)$ defined by putting $d(x) := \T{cl}([D,x])$ so that $d(x)$ agrees with the closure of the commutator $[D,x] : \T{Dom}(D) \to H$. This construction does in turn yield the slip-norm
\[
L_D : \C A \to [0,\infty) \quad L_D(x) := \| d(x) \| .
  \]
  In this context, the extended metric $\mk_{L_D}$ on the state space of $A$ is often referred to as the \emph{Connes metric}.

  The terminology ``spectral metric space'' introduced here below comes from the paper \cite{BMR:DSS}. Let us however emphasize that there are many examples of unital spectral triples which are not spectral metric spaces. It can even happen that the kernel of $L_D$ agrees with the scalars while some states have distance infinity with respect to the Connes metric, see \cite{KyNe:FMS} for an elementary example of this phenomenon. For a more geometric example we refer to \cite{CAMW:SDM} where a similar behaviour is exhibited for the Moyal plane (even though this is in the non-compact setting).
  
    \begin{dfn}\label{d:specmet}
We say that the unital spectral triple $(\C A,H,D)$ is a \emph{spectral metric space} if $(\C A,L_D)$ is a compact quantum metric space.
  \end{dfn}

  Let us clarify that the unital $*$-subalgebra $\C A \su A$ appearing in the unital spectral triple $(\C A,H,D)$ can be enlarged in a substantial fashion. Indeed, instead of $\C A$ we could consider the \emph{Lipschitz algebra} $\Lip_D(A) \su A$ defined as follows: An element $x \in A$ belongs to $\Lip_D(A)$ if and only if the supremum
  \begin{equation}\label{eq:suplip}
\sup\big\{ \big| \inn{D \xi, x \eta} - \inn{\xi,x D \eta}\big| \mid \xi , \eta \in \T{Dom}(D) \, , \, \, \| \xi \|_2 , \| \eta \|_2 \leq 1 \big\} 
\end{equation}
is finite, see \cite[Theorem 3.8]{Chr:WDO} for a number of equivalent conditions. We record that $\big( \Lip_D(A),H,D \big)$ is a unital spectral triple and we therefore obtain the slip-norm $L_D : \Lip_D(A) \to [0,\infty)$. To avoid confusion we often denote this latter slip-norm by $L_D^{\max}$ and notice that $L_D^{\max}(x) = L_D(x)$ for all $x \in \C A$. Likewise, we often apply the notation $d_{\max} : \Lip_D(A) \to \B B(H)$ for the closed derivation coming from the enlarged unital spectral triple. Remark that $L_D^{\max}(x)$ agrees with the supremum in \eqref{eq:suplip} for $x \in \Lip_D(A)$ and that this supremum in turn agrees with the operator norm of $d_{\max}(x)$. 
  
  It is important to keep in mind that $\Lip_D(A) \su A$ can very well be different from the domain of the closure $\ov{d}$ of the derivation $d : \C A \to \B B(H)$, but it always holds that $\T{Dom}(\, \ov{d} \, ) \su \Lip_D(A)$. In particular we get from Theorem \ref{t:approx} that if $(\Lip_D(A),H,D)$ is a spectral metric space, then $(\C A,H,D)$ is a spectral metric space. We expect that the converse is not true but at the time of writing we are not aware of any counter examples. 

\subsection{Quantum $SU(N)$}\label{sec:suqn}
We are now going to review some of the main concepts pertaining to the quantized version of the special unitary group. Our main reference on these matters is the book \cite{KlSc:QGR} and in particular the chapters 6.1 and 9.2. For completeness we also refer the reader to the foundational papers on quantized enveloping algebras, \cite{Dri:QG,Jim:QYB} and $C^*$-algebraic compact quantum groups, \cite{Wor:CMP, Wor:CQG}.   

Let us start out by introducing the universal unital $\cc$-algebra $\C O(M_q(N))$ generated by $N^2$ elements $u_{ij}$, $i,j \in \inn{N}$ subject to the quadratic relations
\begin{equation}\label{eq:quadsuN}
\begin{split}
u_{ik} u_{jk} & = q \cd u_{jk} u_{ik} \, \, , \, \, \, u_{ki} u_{kj} = q \cd u_{kj} u_{ki} \quad \T{for } i < j \\
[u_{il}, u_{jk}] & = 0 \, \, , \, \, \, [u_{ik}, u_{jl}] = (q - q^{-1}) \cd u_{il} u_{jk} \quad \T{for } i < j \, \, , \, \, \, k < l . 
\end{split}
\end{equation}

Whenever we have two non-empty subsets $I,J \su \inn{N}$ with the same number of elements, say $n$, we may choose $1 \leq i_1 < \ldots < i_n \leq N$ and $1 \leq j_1 < \ldots < j_n \leq N$ such that $I = \{i_1,\ldots,i_n\}$ and $J = \{j_1,\ldots,j_n\}$. The quantum $n$-minor determinant with respect to the subsets $I$ and $J$ is then the element in $\C O(M_q(N))$ defined by
\[
D_{IJ} := \sum_{\si \in S_n} (-q)^{\io(\si)} u_{i_{\si(1)} j_1} \clc u_{i_{\si(n)} j_n} ,
\]
where $S_n$ denotes the group of permutations of the set $\inn{n}$ and $\io(\si)$ is the number of inversions in a permutation $\si \in S_n$. For $I = J = \inn{N}$ we refer to $D_{IJ}$ as the \emph{quantum determinant}. Moreover, for each $i,j \in \inn{N}$ we define the \emph{cofactor}
\[
\wit{u}_{ij} := (-q)^{i - j} D_{ \inn{N} \sem \{j\}, \inn{N} \sem \{i\} } . 
\]

The notation $\C O(SL_q(N))$ refers to the unital $\cc$-algebra obtained as the quotient of $\C O(M_q(N))$ by the ideal generated by $D_{ \inn{N},\inn{N} } - 1$ (so inside $\C O(SL_q(N))$, the quantum determinant equals one). It follows from \cite[Chapter 9.2.3, Proposition 10]{KlSc:QGR} that $\C O(SL_q(N))$ becomes a Hopf algebra with coproduct, counit and antipode given on generators by
\[
\De(u_{ij}) := \sum_{k = 1}^N u_{ik} \ot u_{kj}\, , \quad \epsilon(u_{ij}):= \de_{ij}\, , \quad S(u_{ij}) := \wit{u}_{ij} .
\]

As in \cite[Chapter 9.2.4]{KlSc:QGR}, the \emph{coordinate algebra for quantum $SU(N)$} is defined as the Hopf $*$-algebra $\C O(SU_q(N))$ which agrees with $\C O(SL_q(N))$ as a Hopf algebra, but with involution defined on generators by
\[
u_{ij}^* := S(u_{ji}) = \wit{u}_{ji} .
\]

For later use it is important to record the following further relations inside $\C O(SU_q(N))$, see \cite[Proposition 8, Chapter 9.2.2]{KlSc:QGR}:
\begin{equation}\label{eq:unitrans}
\begin{split}
  & \sum_{k = 1}^N u_{ki}^* u_{kj} = \de_{ij} = \sum_{k = 1}^N u_{ik} u_{jk}^* \\
  & \sum_{k = 1}^N q^{2(k - j)} \cd u_{ki} u_{kj}^* = \de_{ij} = \sum_{k = 1}^N q^{2(i - k)} \cd u_{ik}^* u_{jk} .
\end{split}
\end{equation}

Letting $u \in \B M_N\big(\C O(SU_q(N))\big)$ denote the $(N \ti N)$-matrix with entries $u_{ij}$, $i,j \in \inn{N}$, we see from \eqref{eq:unitrans} that $u$ is unitary, and we refer to this unitary matrix as the \emph{fundamental unitary}. 
\medskip

We also introduce the \textit{quantized enveloping algebra of $\G{su}(N)$} defined as the universal unital $*$-algebra $\C U_q(\G{su}(N))$ generated by $3\ell$ elements $E_i,K_i,K_i^{-1}$, $i \in \inn{\ell}$ subject to the following relations for all $i,j \in \inn{\ell}$:
\[
\begin{split}
& K_i K_j = K_j K_i \, \, , \, \, \, K_i K_i^{-1} = 1 = K_i^{-1} K_i \, \, , \, \, \, K_i = K_i^* \\
& K_i E_j = q^{\de_{ij} - \frac{1}{2} \de_{i,j-1} - \frac{1}{2} \de_{i,j+1}} E_j K_i \\
  & [E_i,E_j^*] = \de_{ij} \cd \frac{K_i^2 - K_i^{-2}}{q - q^{-1}} \\
  & [E_i,E_j] = 0 \quad |i-j| > 1 \\
& E_i^2 E_j - (q + q^{-1}) E_i E_j E_i +  E_j E_i^2 = 0 \quad |i - j | = 1 .
\end{split}
\]
We put $F_i := E_i^*$. The unital $*$-algebra $\C U_q(\G{su}(N))$ becomes a Hopf $*$-algebra with coproduct, counit and antipode determined by the formulae
\[
\begin{split}
& \De(K_i) := K_i \ot K_i \, \, , \, \,  \, \De(E_i) := E_i \ot K_i + K_i^{-1} \ot E_i \\
& \epsilon(K_i) := 1 \, \,  , \, \, \, \epsilon(E_i) := 0 \\ 
& S(K_i) := K_i^{-1} \, \, , \, \, \, S(E_i) := -q E_i . 
\end{split}
\]
Remark that $S(F_i) = S(E_i^*) = S^{-1}(E_i)^* = - q^{-1} F_i$. For more details we refer the reader to \cite[Chapter 6.1.2]{KlSc:QGR} -- notice here that Klimyk and Schm\"udgen apply the notation $\breve{U}_q(su_N)$ for the Hopf $*$-algebra we are denoting by $\C U_q(\G{su}(N))$.  


The quantized enveloping algebra can be represented on the Hilbert space $\cc^N$ via the unital $*$-homomorphism $\pi : \C U_q( \G{su}(N)) \to \B B(\cc^N)$ defined on generators by the formulae
\[
\pi(E_i) := e_{i+1,i} \, \, \T{ and } \, \, \, \pi(K_i) := \sum_{j = 1}^N e_{jj} \cd q^{\frac{1}{2} ( \de_{j,i+1} - \de_{j,i})} .
\]
Letting $e_1,\ldots,e_N$ denote the standard orthonormal basis for $\cc^N$ we obtain the functionals $\pi_{ij} : \C U_q(\G{su}(N)) \to \cc$ given by $\pi_{ij}(\eta) := \inn{e_i,\pi(\eta) e_j}$. These functionals give rise to a dual pairing of Hopf $*$-algebras as made explicit in the next theorem. We are presenting an important detail on the proof since it is less easy to find in \cite{KlSc:QGR} even though the result is very much related to \cite[Theorem 18, Chapter 9.4]{KlSc:QGR}. 

\begin{theorem}\label{t:dualpair}
 There exists a unique dual pairing of Hopf $*$-algebras $\inn{\cd,\cd} : \C U_q(\G{su}(N)) \ti \C O( SU_q(N)) \to \cc$ such that
  \[
\inn{\eta,u_{ij}} = \pi_{ij}(\eta) \quad i,j \in \inn{N} .
\]
\end{theorem}
\begin{proof}
  Define the operator $\widehat{R} : \cc^N \ot \cc^N \to \cc^N \ot \cc^N$ by putting
  \[
  \widehat{R} := q \cd \sum_{i = 1}^N e_{ii} \ot e_{ii} + (q - q^{-1}) \cd \sum_{1 \leq i < j \leq N} e_{ii} \ot e_{jj}
  + \sum_{i \neq j} e_{ij} \ot e_{ji} .
  \]
  It can then be verified that $\widehat{R} \cd (\pi \ot \pi) \De(\eta) = (\pi \ot \pi) \De(\eta) \cd \widehat{R}$ for all $\eta \in \C U_q(\G{su}(N))$. This identity implies that the functionals $\pi_{ij} : \C U_q(\G{su}(N)) \to \cc$ satisfy the relations in \eqref{eq:quadsuN} inside the dual Hopf $*$-algebra $\C U_q(\G{su}(N))^\ci$ (with $\pi_{ij}$ replacing $u_{ij}$). We therefore get a unital algebra homomorphism $\Phi : \C O(M_q(N)) \to \C U_q(\G{su}(N))^\ci$ and, applying the argument from \cite[Theorem 18, Chapter 9.4]{KlSc:QGR}, it follows that $\Phi$ vanishes on the element $D_{\inn{N},\inn{N}} - 1 \in \C O(M_q(N))$. It is then not difficult to see that the induced map $\Phi : \C O(SU_q(N)) \to \C U_q(\G{su}(N))^\ci$ is in fact a homomorphism of Hopf $*$-algebras and this ends the proof of the theorem.
\end{proof}

The dual pairing from Theorem \ref{t:dualpair} gives rise to a left action of $\C U_q(\G{su}(N))$ on $\C O(SU_q(N))$ (by linear endomorphisms) defined by 
\begin{equation}\label{eq:actenv}
d_\eta (x)   := (\inn{S^{-1}(\eta) ,\cdot} \ot 1)\Delta(x)  
\end{equation}
for $x\in \mathcal{O}(SU_q(N)) $ and $\eta \in \C U_q(\G{su}(N))$. We record that
\begin{equation}\label{eq:actgener}
\begin{split}
  d_{K_r}(u_{ij})&= \inn{K_r^{-1}, u_{ii}} \cd u_{ij}  = q^{\frac{1}{2}\de_{ir} - \frac{1}{2} \de_{i,r+1}} \cd u_{ij} \\
  d_{E_r}(u_{ij}) & = -q^{-1} \inn{E_r, u_{i,i-1}} \cd u_{i-1,j} = -q^{-1} \de_{i,r+1} \cd u_{i-1,j} \\
  d_{F_r}(u_{ij}) & = -q \inn{F_r,u_{i,i+1}} \cd u_{i+1,j} = - q \de_{ir} \cd u_{i+1,j}
\end{split}
\end{equation}
for all $r \in \inn{\ell}$ and $i,j \in \inn{N}$. Furthermore, it follows from the defining properties of a dual pairing that we have the identities
\begin{equation}\label{eq:calcrule}
d_\eta(xy) = d_{\eta_{(2)}}(x)d_{\eta_{(1)}}(y)  \quad \T{and} \quad d_\eta(x^*)=d_{S(\eta^*)}(x)^* 
\end{equation}
for all $x,y\in \C O(SU_q(N))$ and $\eta \in  \C U_q(\G{su}(N))$ (where we apply the Sweedler notation $\Delta(\eta)=\eta_{(1)}\otimes \eta_{(2)}$). In order to avoid the flip in the first identity in \eqref{eq:calcrule} one may replace it by the rule $d_{S(\eta)}(x y) = d_{S(\eta_{(1)})}(x) d_{S(\eta_{(2)})}(y)$. 

Let $i \in \inn{\ell}$. As a consequence of the above considerations we get that $d_{K_i}$ is an algebra automorphism of $\C O(SU_q(N))$ with inverse $d_{K_i}^{-1}=d_{K_i^{-1}}$. It moreover holds that $d_{E_i}$ and $d_{F_i}$ are twisted derivations of $\C O(SU_q(N))$, meaning that they satisfy the twisted Leibniz rules
\begin{equation}\label{eq:twileib}
\begin{split}
d_{E_i}(xy)&= d_{E_i}(x)d_{K_i^{-1}}(y)+ d_{K_i}(x) d_{E_i}(y) \quad \T{and} \\
d_{F_i}(xy)&= d_{F_i}(x)d_{K_i^{-1}}(y)+d_{K_i}(x) d_{F_i}(y)
\end{split}
\end{equation}
for all $x,y \in \C O(SU_q(N))$. The behaviour of the operations $d_{E_i}$, $d_{F_i}$ and $d_{K_i}$ with respect to the involution is summarized by the formulae 
\[
d_{E_i}(x^*)=-q^{-1}d_{F_i}(x)^* \, \, , \, \, \, d_{F_i}(x^*)=-q d_{E_i}(x)^* \, \, \T{ and } \, \, \, d_{K_i}(x^*)=d_{K_i}^{-1}(x)^* .
\]

We define \textit{quantum} $SU(N)$ -- denoted $C(SU_q(N))$ -- as the universal $C^*$-algebra generated by $\C O(SU_q(N))$. The comultiplication extends to a unital $*$-homomorphism $\Delta : C(SU_q(N))\to C(SU_q(N))\ot_{\T{min}} C(SU_q(N)) $ (applying the minimal tensor product), and in this way $C(SU_q(N))$ is turned into a ($C^*$-algebraic) compact quantum group, see \cite{Wor:CMP, Wor:CQG}. It also follows by construction that the counit extends to a unital $*$-homomorphism $\epsilon : C(SU_q(N))\to \B C$. 

Let us reserve the notation $h : C(SU_q(N))\to \B C$ for the Haar state of the compact quantum group $C(SU_q(N))$ and recall that this state is uniquely determined by the identity $h(1)=1$ together with the bi-invariance
\begin{equation}\label{eq:haar}
(1\otimes h)\Delta(x)=h(x)\cd 1=(h\otimes 1)\Delta(x) \quad \T{for all } x\in C(SU_q(N)) ,
\end{equation}
see \cite[Theorem 1.3]{Wor:CQG}. As a consequence of the bi-invariance property we get that
\begin{equation}\label{eq:calcrulehaar}
h  \circ d_\eta =\epsilon(\eta)\cdot  h \quad \T{for all } \eta \in \C U_q(\G{su}(N)).
\end{equation}

Let us denote the modular automorphism of the Haar state by $\theta : \C O(SU_q(N))\to \C O(SU_q(N))$, see \cite[Theorem 5.6]{Wor:CQG}. Record in this respect that $\te(x^*) = \te^{-1}(x)^*$ and that $h(xy)=h(y\theta(x))$ for all $x,y\in \C O(SU_q(N))$. This modular automorphism can be computed explicitly on the generators for $\C O(SU_q(N))$. By \cite[Proposition 34 and Example 9, Chapter 11.3.4]{KlSc:QGR} it holds that
\begin{equation}\label{eq:modular}
\te(u_{ij}) = q^{2(i + j - N - 1)} u_{ij} \quad i,j \in \inn{N} .
\end{equation}
Let $L^2(SU_q(N))$ denote the Hilbert space associated with the GNS-construction applied to the Haar state. We know from \cite[Theorem 1.1]{Nag:HQG} that the Haar state is faithful and it follows that both the GNS-representation $\pi : C(SU_q(N)) \to \B B(L^2(SU_q(N)))$ and the canonical linear map $\io : C(SU_q(N)) \to L^2(SU_q(N))$ are injective. We recall that inner products are linear in the second variable so that $\inn{ \io(x),\io(y) } = h(x^* y)$ for all $x,y \in C(SU_q(N))$.

\subsection{Peter-Weyl decomposition and Schur orthogonality}\label{ss:weylschur}
In this subsection we briefly remind the reader of the Peter-Weyl decomposition of the coordinate algebra $\C O(SU_q(N))$ and a relevant consequence of the Schur orthogonality relations. These results are in fact available in the more general setting of cosemisimple Hopf algebras, see \cite[Chapter 11.2]{KlSc:QGR}, but we are here focusing on the specific example $\C O(SU_q(N))$.

Recall that a \emph{corepresentation} of $\C O(SU_q(N))$ is a linear map $v : \cc^n \to \cc^n \ot \C O(SU_q(N))$ satisfying the identities
\[
(1 \ot \De) v = (v \ot 1) v \, \, \T{ and } \, \, \, (1 \ot \epsilon) v = 1 .
\]
A corepresentation $v$ acting on $\cc^n$ (with standard orthonormal basis $\{e_k\}_{k = 1}^n$) gives rise to the \emph{matrix coefficients} $v_{ij}$ labelled by $i,j \in \{1,\ldots,n\}$ and satisfying that
\[
v(e_j) = \sum_{i = 1}^n e_i \ot v_{ij}  \quad \T{for all } j \in \inn{n} .
\]
Our corepresentation $v$ is called \emph{unitary}, if $v_{ij}^* = S(v_{ji})$ for all $i,j \in \inn{n}$. We say that a corepresentation $v$ is \emph{irreducible}, if it is impossible to find a proper subspace $U \su \cc^n$ such that
\[
v(U) \su U \ot \C O(SU_q(N)) .
\]
Two irreducible corepresentations $v$ and $w$ (both acting on $\cc^n$) are said to be \emph{equivalent}, if there exists an invertible linear map $T : \cc^n \to \cc^n$ such that $v T = (T \ot 1) w$. Letting $\widehat{\C O(SU_q(N))}$ denote the set of equivalence classes of irreducible unitary corepresentations we choose a representative $v^\al : \cc^{n_\al} \to \cc^{n_\al} \ot \C O(SU_q(N))$ for each $\al \in \widehat{\C O(SU_q(N))}$. 

The \emph{Peter-Weyl decomposition} in this setting says that the set of matrix coefficients
\[
\big\{ v_{ij}^\al \mid \al \in \widehat{\C O(SU_q(N))} \, , \, \, i,j \in \inn{n_\al} \big\}
\]
forms a vector space basis for $\C O(SU_q(N))$, see \cite[Theorem 13, Chapter 11.2.1]{KlSc:QGR}. We emphasize here that the entries $u_{ij}$ for $i,j \in \inn{N}$ of the fundamental unitary $u$ are in fact matrix coefficients of a unitary corepresentation and it therefore follows that $\C O(SU_q(N))$ is a compact matrix quantum group algebra in the sense of \cite[Definition 10, Chapter 11.3.1]{KlSc:QGR}.

For each $\al \in \widehat{\C O(SU_q(N))}$ we define the finite dimensional subspace
\[
\C C(\al) := \T{span}\big\{ v_{ij}^\al \mid i,j \in \inn{n_\al} \big\} \su \C O(SU_q(N)) 
\]
and record that $\C C(\al)$ is in fact a coalgebra in its own right since $\De(\C C(\al)) \su \C C(\al) \ot \C C(\al)$. The \emph{Schur orthogonality relations} then imply that these subspaces are mutually orthogonal meaning that
\[
h( x^* y ) = 0 
\]
whenever $x \in \C C(\al)$ and $y \in \C C(\be)$ for some $\al \neq \be$, see \cite[Proposition 15, Chapter 11.2.2]{KlSc:QGR}.

\subsection{Quantum spheres and quantum projective spaces}\label{ss:quasph}
In this subsection we recall the definition of the Vaksman-Soibelman quantum spheres and the associated quantum projective spaces introduced in \cite{VaSo:AQS}. Inside the quantum projective space (of complex dimension $\ell$) we pay particular attention to a small commutative unital $C^*$-subalgebra which is isomorphic to the continuous functions on the quantized interval $I_q := \big\{ q^{2m} \mid m \in \nn_0\big\} \cup \{0\}$. The quantized interval is going to play a pivotal role in our approach to the spectral metric structure of quantum projective spaces. 

For each $i \in \inn{N}$ we put $z_i := u_{Ni}$ so that $z_i$ agrees with the $i^{\T{th}}$ entry in the last row of the fundamental unitary $u$. We let $\C O(S_q^{2\ell + 1})$ denote the smallest unital $*$-subalgebra of $\C O(SU_q(N))$ such that $z_i$ belongs to $\C O(S_q^{2\ell + 1})$ for all $i$ in $\inn{N}$. This unital $*$-subalgebra is referred to as the \emph{coordinate algebra for the quantum sphere}. The following identities are consequences of the defining relations for the coordinate algebra for quantum $SU(N)$: 
\begin{equation}\label{eq:defsphere}
\begin{split}
& z_iz_j =qz_j z_i  \quad i < j \, \, , \, \, \, z_i^* z_j =qz_jz_i^* \quad i \neq j \\ 
&  [z_i^*,z_i]= (1-q^2) \sum_{j = 1}^{i - 1} z_jz_j^* \quad i > 1 \\
& [z_1^*,z_1] = 0 \, \, , \, \, \, \sum_{i = 1}^N z_iz_i^* =1 . 
 \end{split}
\end{equation}
We let $\C O(\B CP_q^{\ell} )$ denote the smallest unital $*$-subalgebra of $\C O(S_q^{2\ell + 1})$ such that $z_i z_j^* \in \C O(\B CP_q^\ell)$ for all $i,j \in \inn{N}$. This unital $*$-subalgebra is called the \emph{coordinate algebra for the quantum projective space}.


%

We define the unital $C^*$-algebras $C(S_q^{2\ell + 1})$ and $C(\B CP_q^\ell)$ as the $C^*$-closures of $\C O(S_q^{2\ell + 1})$ and $\C O(\B CP_q^\ell)$ inside $C(SU_q(N))$. We refer to $C(S_q^{2\ell + 1})$ and $C(\B CP_q^\ell)$ as the \textit{quantum sphere} and the \textit{quantum projective space}, respectively.

The coproduct $\De$ on the coordinate algebra for quantum $SU(N)$ induces a right coaction on both the coordinate algebra for the quantum sphere and the quantum projective space. These coactions are denoted by
\[
\de : \C O(S_q^{2\ell + 1}) \to \C O(S_q^{2\ell + 1}) \ot \C O(SU_q(N)) \, \, \T{ and } \, \, \,
\de : \C O(\B CP_q^{\ell}) \to \C O(\B CP_q^{\ell}) \ot \C O(SU_q(N)) 
\]
and they extend by continuity to right coactions of the compact quantum group $C(SU_q(N))$ on the quantum sphere and the quantum projective space. 
\medskip

We introduce the elements $x_i := z_i z_i^*$ for $i \in \inn{N}$ together with the sums
\[
y_j := \sum_{i=1}^{j} x_i =  \sum_{i = 1}^j z_i z_i^* \quad j \in \inn{N} . 
\]
It is relevant to analyse the $C^*$-subalgebra $C^*(z_N z_N^*, z_N^* z_N)$ of quantum projective space generated by the elements $z_Nz_N^*$ and $z_N^* z_N$. Since $z_N z_N^* = 1 - y_\ell$ and $z_N^* z_N = 1 - q^2 y_\ell$ we get that $C^*(z_N z_N^*, z_N^* z_N)$ is $*$-isomorphic to the continuous functions on the spectrum of $y_\ell$. An application of \cite[Lemma 4.2]{HoSz:QSPS} (see also \cite[Proposition 2.11]{MiKa:HSV}) shows that
\[
\Sp(y_\ell) = \big\{q^{2m}\mid m\in \B N_0 \big\} \cup \{0\}. 
\]
As mentioned earlier, we refer to this subset of $\B R$ as the \emph{quantized interval} and apply the notation $I_q := \Sp(y_\ell)$ so that $C^*(z_Nz_N^*,z_N^*z_N) \cong C(I_q)$ via Gelfand duality. We shall from now on identify $C(I_q)$ with a unital $C^*$-subalgebra of $C(\B CP_q^\ell)$. To be explicit, this identification sends the inclusion $I_q \to \B C$ to the element $y_\ell$.

The indicator function for the point $q^{2k} \in \B R$ restricts to a continuous function on $\Sp(y_\ell)$ as soon as $k \in \nn_0$. The corresponding projection in $C(I_q)$ (obtained via continuous functional calculus) is denoted by $p_k$. For $k = 0$, we may describe $p_0$ as the limit $p_0 =\lim_{n\to \infty} y_\ell^n$ and for $k > 0$ it can be verified that
\begin{equation}\label{eq:mproj}
  p_k =\lim_{n\to \infty} \left(\frac{1}{q^{2kn}}y_\ell^n -\sum_{i=1}^k \frac{1}{q^{2in}}p_{k - i} \right) .
\end{equation}
We apply the convention that $p_k := 0$ for $k \in \zz \sem \nn_0$.

The unital $*$-subalgebra of $\C U_q(\G{su}(N))$ generated by $K_i$, $K_i^{-1}$ and $E_i$ for $i \in \inn{\ell - 1}$ is denoted by $\C U_q(\G{su}(\ell))$ (for $\ell = 1$ this is just $\cc 1$ by convention). It is then important to clarify that
\begin{equation}\label{eq:homsph}
d_\eta(x) = \epsilon(\eta) \cd x \quad \T{for all } x \in \C O(S_q^{2\ell + 1}) \T{ and } \eta \in \C U_q(\G{su}(\ell)) .
\end{equation}
In the case where $x \in \C O(\B C P_q^\ell) \su \C O(S_q^{2\ell + 1})$ it moreover holds that
\begin{equation}\label{eq:hompro}
d_{K_\ell}(x) = x .
\end{equation}

\section{Spectral geometry of quantum projective spaces}\label{s:specgeo}
In this section, we review the spectral geometry of quantum projective spaces as witnessed by the unital spectral triple introduced by D'Andrea and D\k{a}browski in \cite{DaDa:DQP} building on the earlier work \cite{DaDaLa:NGQPP} and \cite{DaSi:DSP} where the lower dimensional cases are treated. It should be emphasized that related constructions appear in many places, see e.g. \cite{DaBuSo:DDQ,KrTu:DDQ,Mat:DDQ}, but in the present text we follow the original approach of D'Andrea and D\k{a}browski closely. 

We start out with a small well-known result relating the adjoint operation on $\C U_q(\G{su}(N))$ to the adjoint operation of closable unbounded operators on the Hilbert space $L^2(SU_q(N))$. As usual it is convenient to apply Sweedler notation for coproducts.

\begin{lemma}\label{l:ip}
For every $x,y\in \mathcal{O}(SU_q(N))$ and $\xi \in \C U_q(\G{su}(N))$ it holds that
\[
\inn{d_{\xi^*}(x) , y } = \inn{ x,d_\xi(y)} . 
\]
\end{lemma}
\begin{proof}
Applying \eqref{eq:calcrule} and \eqref{eq:calcrulehaar} we compute as follows:
\[
\begin{split}
h(x^* d_\xi(y)) & = \epsilon(\xi_{(1)}) \cd h(x^* d_{\xi_{(2)}}(y)) = h\big( d_{S(\xi_{(1)})}(x^* \cd d_{\xi_{(2)}}(y)) \big) \\
& = h\big( d_{S(\xi_{(1)})}(x^*) \cd d_{S(\xi_{(2)}) \xi_{(3)}}(y) \big)
= h( d_{S(\xi)}(x^*) y) = h(d_{\xi^*}(x)^* y) . \qedhere
\end{split}
\]
\end{proof}

In order to define a $q$-deformed analogue of the antiholomorphic forms over quantum projective space we consider the usual exterior algebra $\La(\cc^\ell) = \op_{k = 0}^\ell \La^k(\cc^\ell)$. We view $\La(\cc^\ell)$ as a finite dimensional Hilbert space with orthonormal basis $\{ e_I\}_{I \su \inn{\ell}}$ indexed by the power set of $\inn{\ell} = \{1,\ldots,\ell\}$. 

The operation $\sharp$ counts the number of elements in a finite (or empty) set. Hence, for $I \su \nn$ and $j \in \nn$ we have that $\sharp ( I \cap \inn{j})$ is equal to the number of elements in $I$ which are less than or equal to $j$. We shall also apply the variant over the Kronecker delta such that $\de_{j,I}$ is equal to $1$ for $j \in I$ and $0$ for $j \notin I$.

For each $j \in \inn{\ell}$ we are interested in the $q$-deformed exterior multiplication operator
$\ep_j^q : \La(\cc^\ell) \to \La(\cc^\ell)$ determined by
\[
\ep_j^q(e_I) := \fork{ccc}{0 & \T{for} & j \in I \\ (-q)^{-\sharp (I \cap \inn{j})} e_{I \cup \{ j \} }  & \T{for} & j \notin I} .
\]
It is important to record the commutation rule $\ep_i^q \ep_j^q = -q \ep_j^q \ep_i^q$ for all $i < j$. 

Recall that $\C U_q(\G{su}(\ell)) \su \C U_q(\G{su}(N))$ denotes the unital $*$-subalgebra of $\C U_q(\G{su}(N))$ generated by the elements $K_r,K_r^{-1},E_r$ for $r < \ell$ (for $\ell = 1$ this is just the scalars $\cc 1$). It clearly holds that $\C U_q(\G{su}(\ell))$ is a Hopf $*$-subalgebra of $\C U_q(\G{su}(N))$.

We introduce a unital $*$-homomorphism $\si : \C U_q(\G{su}(\ell)) \to \B B\big( \La(\cc^\ell) \big)$ in the following fashion: For $r < \ell$ and $I \su \inn{\ell}$ define
\[
\begin{split}
\si(E_r)(e_I) & := \fork{ccc}{ e_{ (I \sem \{r+1\}) \cup \{r\}} & \T{for} & r + 1 \in I \T{ and } r \notin I \\ 
0 & & \T{else} } \quad \T{and} \\
\si(K_r)(e_I) & := q^{\frac{1}{2} (\de_{r,I} - \de_{r + 1, I})} e_I . 
  \end{split}
\]
It is also convenient to record the formula
\[
\si(F_r)(e_I) = \fork{ccc}{ e_{ (I \sem \{r\}) \cup \{r+1\}} & \T{for} & r \in I \T{ and } r + 1 \notin I \\ 
0 & & \T{else} } .
\]
We emphasize that $\si$ restricts to a unital $*$-homomorphism $\si_k : \C U_q(\G{su}(\ell)) \to \B B\big( \La^k(\cc^\ell) \big)$ for all $k \in \{0,1,\ldots,\ell\}$ and that both $\si_0$ and $\si_\ell$ agree with the counit $\epsilon$ (upon identifying $\B B(\cc)$ with $\cc$).
%

Let us clarify the compatibility between the representation $\si$ and the $q$-deformed exterior multiplication operators as investigated in \cite[Proposition 3.9]{DaDa:DQP}. For each vector $v  \in \cc^{\ell} = \La^1(\cc^\ell)$ we define
\[
\ep_v^q := \sum_{j = 1}^\ell \inn{e_j,v} \cd \ep^q_j .
\]

\begin{lemma}\label{l:extsig}
It holds that
\[
\si(\eta) \cd \ep^q_v = \ep^q_{\si(\eta_{(1)})(v)} \cd \si(\eta_{(2)})
\quad \mbox{for all } \eta \in \C U_q(\G{su}(\ell)) \mbox{ and } v \in \cc^\ell .
\]
\end{lemma}

Let us from now on fix an integer $M \in \zz$ which is going to index a twist of the antiholomorphic forms by a finitely generated projective module over the coordinate algebra $\C O( \B C P_q^\ell)$. Notice that the incorporation of this twist is necessary in order to recover the D\k{a}browski-Sitarz spectral triple over the Podle\'s sphere for $\ell = 1$, see Subsection \ref{ss:dabsit}. 

For each $I \su \inn{\ell}$ we define the subspace 
\[
\C O(SU_q(N))^I_M := \big\{ x \in \C O(SU_q(N)) \mid d_{K_\ell}(x) = q^{\frac{M - \sharp I}{2} - \frac{1}{2} \de_{\ell,I}} x \big\} 
\]
of the coordinate algebra $\C O(SU_q(N))$. For each $k \in \{0,1,\ldots, \ell \}$ we then introduce the subspace
\[
\Ga_M^k := \T{span}_{\cc}\big\{ x \ot e_I \mid I \su \inn{\ell} \, , \, \, \sharp I = k \, , \, \, x \in \C O(SU_q(N))^I_M \big\} 
\]
of $\C O(SU_q(N)) \ot \La^k(\cc^\ell)$. 

\begin{dfn}\label{d:antihol}
 Let $k \in \{0,1,\ldots,\ell\}$ and $M \in \zz$. The \emph{twisted antiholomorphic forms} of degree $k$ are defined as the subspace
  \[
\Om_M^k := \big\{ \om \in \Ga_M^k \mid d_{\eta_{(1)}} \ot \si(\eta_{(2)}) \om = \epsilon(\eta) \cd \om \, \, \T{for all } \eta \in \C U_q(\G{su}(\ell)) \big\} 
\]
of the algebraic tensor product $\C O(SU_q(N)) \ot \La^k(\cc^\ell)$.
\end{dfn}

Let us specify that $\Om^k_M$ carries a left action of the coordinate algebra $\C O(\B CP_q^\ell)$ which is inherited from the algebra structure of $\C O(SU_q(N))$. Indeed, because of \eqref{eq:homsph} and \eqref{eq:hompro} for every $\om \in \Om^k_M$ and $x \in \C O(\B CP_q^\ell)$ it holds that $(x \ot 1) \cd \om \in \Om^k_M$. We shall often view $\Om^k_M$ as a subspace of the Hilbert space tensor product $L^2( SU_q(N)) \ot \La^k(\cc^\ell)$ via the inclusion
\[
\io \ot 1 : \C O(SU_q(N)) \ot \La^k(\cc^\ell) \to L^2(SU_q(N)) \ot \La^k(\cc^\ell)  
\]
and the corresponding Hilbert space completion of $\Om^k_M$ is denoted by $L^2(\Om^k_M,h)$. The left action of $\C O(\B CP_q^\ell)$ on $\Om^k_M$ then induces a unital $*$-homomorphism $\rho^k : C(\B CP_q^\ell) \to \B B\big( L^2(\Om^k_M,h)\big)$. Putting $\Om_M := \op_{k = 0}^\ell \Om^k_M$ and $L^2(\Om_M,h) := \op_{k = 0}^\ell L^2(\Om^k_M,h)$, we obtain an injective unital $*$-homomorphism
\begin{equation}\label{eq:repqua}
\rho : C(\B CP_q^\ell) \to \B B\big( L^2(\Om_M,h) \big) 
\end{equation}
by taking the direct sum of the representations $\rho^k$ on $L^2(\Om^k_M,h)$. Notice that the injectivity of $\rho$ for $M = 0$ relies on the faithfulness of the Haar state (restricted to the quantum projective space $C(\B CP_q^\ell)$). For $M \neq 0$ one also needs the sphere relations $\sum_{i = 1}^N z_i z_i^* = 1 = \sum_{i = 1}^N q^{2(N-i)} z_i^* z_i$ together with the extra observation that $z_j \in \Om_{-1}^0$ and $z_j^* \in \Om_1^0$ for all $j \in \inn{N}$. 

The Hilbert space $L^2(\Om_M,h)$ is equipped with the $\zz/2\zz$-grading with grading operator $\ga : L^2(\Om_M,h) \to L^2(\Om_M,h)$ defined as the restriction of the selfadjoint unitary $\ga \in \B B\big(  L^2(SU_q(N)) \ot \La(\cc^\ell) \big)$ determined by
  \[
  \ga( \xi \ot e_I) := (-1)^{\sharp I} \xi \ot e_I .
  \]
 
  For each $i \in \inn{\ell}$ we introduce the element $M_i \in \C U_q(\G{su}(N))$ by putting $M_\ell := E_{\ell}$ and then recursively define
  \begin{equation}\label{eq:recurs}
M_i := E_i M_{i + 1} - q^{-1} M_{i+1} E_i \quad \T{for } i \in \inn{\ell - 1} .
\end{equation}
We shall also need the element $N_i \in \C U_q(\G{su}(N))$ given by $N_i := K_i K_{i + 1} \clc K_\ell$. With these definitions ready we introduce the linear maps
\[
\ov \pa := \sum_{i = 1}^\ell d_{N_i M_i^*} \ot \ep^q_i \quad \T{and} \quad \ov \pa^\da := \sum_{i = 1}^\ell d_{M_i N_i} \ot (\ep^q_i)^*
\]
which a priori act on the algebraic tensor product $\C O(SU_q(N)) \ot \La(\cc^\ell)$. The next result, which can be found as \cite[Proposition 5.6]{DaDa:DQP}, guarantees that the above linear maps preserve the subspace $\Om_M \su \C O(SU_q(N)) \ot \La(\cc^\ell)$. We give a few details on the proof.

\begin{prop}\label{p:welldef}
  Let $M \in \zz$. The expressions $\ov{\pa}$ and $\ov{\pa}^\da$ determine linear maps on $\Om_M$. It moreover holds that $\ov{\pa}^2 = 0 = (\ov{\pa}^\da)^2$ and
  \[
\inn{ \ov{\pa} \om, \xi} = \inn{ \om, \ov{\pa}^\da \xi} \quad \mbox{for all } \om,\xi \in \Om_M .
  \]
\end{prop}
\begin{proof}
 We focus on showing that the endomorphism $\ov{\pa} : \C O(SU_q(N)) \ot \La(\cc^\ell) \to \C O(SU_q(N)) \ot \La(\cc^\ell)$ preserves the subspace $\Om_M$. Let first $k \in \{0,1,\ldots,\ell - 1 \}$ and let us show that $\ov{\pa}(\Ga_M^k) \su \Ga_M^{k+1}$. For each $i \in \inn{\ell}$ we define $Y_i := N_i M_i^* \in \C U_q(\G{su}(N))$ and record that $K_\ell Y_i = Y_i K_\ell q^{-\frac{1}{2} -\frac{1}{2} \de_{i,\ell}}$. It is then not difficult to see that this identity entails the inclusion $(d_{Y_i} \ot \ep_i^q)(\Ga_M^k) \su \Ga_M^{k+1}$ and hence that $\ov{\pa}(\Ga_M^k) \su \Ga_M^{k+1}$. In order to verify that $\ov{\pa}(\Om^k_M) \su \Om^{k+1}_M$ it now suffices to fix an $r \in \inn{\ell - 1}$ and show that the three commutators
  \begin{equation}\label{eq:preome}
    \begin{split}
&  [d_{K_r} \ot \si(K_r), \ov \pa] \, \, , \, \, \, [d_{E_r} \ot \si(K_r) + d_{K_r^{-1}} \ot \si(E_r), \ov \pa] \quad \T{and} \\
      &   [d_{F_r} \ot \si(K_r) + d_{K_r^{-1}} \ot \si(F_r), \ov \pa]
    \end{split}
  \end{equation}
  vanish. We present the relevant details for the third commutator.

  An application of \cite[Lemma 3.2]{DaDa:DQP} shows that $F_r$ commutes with $Y_i$ whenever $i \neq r,r+1$ and for the remaining cases we have the identities
  \[
F_r Y_r = q^{-1/2} Y_r F_r \quad \T{and} \quad F_r Y_{r + 1} = q^{1/2} ( Y_{r + 1} F_r - K_r^{-1} Y_r ) .  
\]
Similarly, we see from Lemma \ref{l:extsig} that $\si(F_r)$ commutes with $\ep_i^q$ whenever $i \neq r,r+1$ and the remaining cases satisfy the identities
\[
\si(F_r) \ep_r^q = \ep_{r + 1}^q \si(K_r) + q^{-1/2} \ep_r^q \si(F_r) \quad \T{and} \quad \si(F_r) \ep_{r + 1}^q = q^{1/2} \ep_{r+1}^q \si(F_r) .
\]
We also record that $K_r$ commutes with $Y_i$ whenever $i \neq r,r+1$ and that
\[
K_r Y_r = Y_r K_r q^{-1/2} \quad \T{and} \quad K_r Y_{r + 1} = Y_{r + 1} K_r q^{1/2} .
\]
Another application of Lemma \ref{l:extsig} moreover entails that 
\[
\si(K_r) \ep_r^q = q^{1/2} \ep_r^q \si(K_r) \quad \T{and} \quad \si(K_r) \ep_{r+1}^q = q^{-1/2} \ep_{r+1}^q \si(K_r)
\]
whereas $\si(K_r)$ commutes with the remaining $q$-deformed exterior multiplication operators. Combining these formulae, we get that the third commutator in \eqref{eq:preome} is equal to zero. 
\end{proof}

The above proposition allows us to define the symmetric unbounded operator
\[
\C D_q := \ov \pa + \ov{\pa}^\da : \Om_M \to L^2(\Om_M,h) .
\]
The closure is denoted by $D_q := \T{cl}(\C D_q)$. It is not hard to see that $\C D_q$ anticommutes with the grading operator $\ga$, entailing that $D_q$ is odd. The following theorem is part of \cite[Theorem 6.2]{DaDa:DQP} but D'Andrea and D\k{a}browski also treat equivariance, reality and determine the spectral dimension. We present some further discussion on the equivariance condition in the next subsection.  

\begin{theorem}\label{t:dada}
It holds that $(\C O(\B C P_q^\ell),L^2(\Om_M,h),D_q)$ is an even unital spectral triple. 
\end{theorem}

We apply the notation $\Lip_{D_q}(\B C P_q^\ell) \su C( \B C P_q^\ell)$ for the Lipschitz algebra associated with the above unital spectral triple and the corresponding closed derivation is denoted by $d_{\max} : \Lip_{D_q}(\B C P_q^\ell) \to \B B\big( L^2(\Om_M,h) \big)$. As discussed in Section \ref{introCQMS}, the unital spectral triple from Theorem \ref{t:dada} induces two slip-norms
\[
L_{D_q} : \C O(\B C P_q^\ell)\to [0,\infty) \, \, \T{ and } \, \, \, L_{D_q}^{\max} : \Lip_{D_q}(\B C P_q^\ell) \to [0,\infty)  
    \]
which agree on $\C O(\B C P_q^\ell)$. Moreover, we have the expression $L_{D_q}^{\max}(x) := \| d_{\max}(x) \|$ for all $x \in \Lip_{D_q}(\B C P_q^\ell)$. We are in this text almost exclusively focusing on the slip-norm $L_{D_q}^{\max}$.

\subsection{Commutators and equivariance}
In this subsection we continue our review of the noncommutative geometry of quantum projective spaces. More precisely, we discuss the equivariance properties and give an explicit computation of the closed derivation $d_{\max}$ in the case where the input belongs to the coordinate algebra $\C O(\B C P_q^\ell)$. We recall that $M \in \zz$ is a fixed integer.

Let us first consider the restriction of $d_{\max}$ to the coordinate algebra: 


\begin{lemma}\label{l:partial}
  Let $x\in \C O(\B CP_q^\ell)$. The commutators $[\ov{\pa},x \ot 1]$ and $[\ov{\pa}^\da,x \ot 1]$ (initially defined on $\Om_M$) extend to bounded operators on $L^2(\Om_M,h)$ given respectively by the two expressions
  \[
  q^{-1} \sum_{i = 1}^\ell (-q)^{i - \ell} d_{ F_i  \clc F_\ell}(x) \ot \ep^q_i \, \, \mbox{ and } \, \, \,
    \sum_{i = 1}^\ell d_{ E_i  \clc E_\ell}(x) \ot (\ep^q_i)^* .
    \]
\end{lemma}
\begin{proof}
  By \cite[Lemma 5.7]{DaDa:DQP} we have that the commutator $[ \ov{\pa},x \ot 1]$ extends to the bounded operator $\sum_{i = 1}^\ell d_{N_i M_i^*}(x) \ot \ep^q_i$ (which therefore in particular preserves the subspace $\Om_M$). Hence to prove our formula for $[\ov{\pa}, x \ot 1]$ it suffices to show that $d_{N_i M_i^*}(x) = q^{-1} (-q)^{i - \ell} d_{ F_i  \clc F_\ell}(x)$ for all $i \in \inn{\ell}$. Remark now that $N_i M_i^* = q^{-1} M_i^* N_i$ and that $d_{N_i}(x) = x$. We therefore only need to show that $d_{M_i^*}(x) = (-q)^{i - \ell} d_{F_i \clc F_\ell}(x)$. This is straightforward to verify using the recursive definition of $M_i$ from \eqref{eq:recurs} together with the fact that $d_{F_r}(x) = 0$ for all $r \in \inn{\ell - 1}$.

The formula for $[\ov{\pa}^\da, x \ot 1]$ now follows by taking adjoints. Indeed, for $\om \in \Om_M$ we apply \eqref{eq:calcrule} and Proposition \ref{p:welldef} to get that
\[
\begin{split}
& [\ov{\pa}^\da, x \ot 1](\om) = - [\ov{\pa}, x^* \ot 1]^*(\om) = \sum_{i = 1}^\ell (-q)^{i - N} ( d_{ F_i  \clc F_\ell}(x^*)^* \ot (\ep^q_i)^*)(\om) \\
& \quad = \sum_{i = 1}^\ell (-q)^{i - N} ( d_{S( E_\ell \clc E_i)}(x) \ot (\ep^q_i)^*)(\om)
= \sum_{i = 1}^\ell ( d_{E_i \clc E_\ell}(x) \ot (\ep^q_i)^*)(\om) . \qedhere
\end{split}
\]
\end{proof}


Our next objective is to describe the equivariance properties of the unital spectral triple $\big( \Lip_{D_q}(\cc P_q^\ell), L^2(\Om_M,h),D_q\big)$. Since we are here dealing with the Lipschitz algebra $\Lip_{D_q}(\B CP_q^\ell)$ and not just the coordinate algebra $\C O(\B CP_q^\ell)$, the best way of explaining equivariance is to work with the multiplicative unitary which implements the coproduct of quantum $SU(N)$ at the Hilbert space level, see \cite{BaSk:UMD} for more details on multiplicative unitaries. 

Define the endomorphism $W$ of the algebraic tensor product $\C O(SU_q(N)) \ot \C O(SU_q(N))$ by the formula
\[
W(x \ot y) := \De(x)(1 \ot y) .
\]
An application of the Peter-Weyl decomposition and the faithfulness of the Haar state (at the level of the coordinate algebra) shows that $W$ is in fact a linear automorphism of $\C O(SU_q(N)) \ot \C O(SU_q(N))$. The same argument also shows that $W$ induces a unitary operator $W$ on the Hilbert space tensor product $L^2(SU_q(N)) \hot L^2(SU_q(N))$. This unitary operator implements the coproduct on the compact quantum group $C(SU_q(N))$ in so far that
\[
W(z \ot 1)W^* = \De(z) \quad \T{for all } z \in C(SU_q(N)),
\]
where the quantities on both sides are interpreted as bounded operators on $L^2(SU_q(N)) \hot L^2(SU_q(N))$.

In fact, for the purposes of this text it is relevant to consider a more general situation. Let us look at a unital $C^*$-algebra $A$ and assume that $\mu : A \to \B C$ is a fixed faithful state. Let $L^2(A,\mu)$ denote the corresponding Hilbert space coming from the GNS-construction and let $\pi : A \to \B B(L^2(A,\mu))$ denote the associated injective unital $*$-homomorphism. Suppose that $\Phi : C(SU_q(N)) \to A$ is a surjective unital $*$-homomorphism and put $\C A := \Phi\big( \C O(SU_q(N))\big)$ so that $\C A \su A$ is a norm-dense unital $*$-subalgebra. Instead of $W$ we may then consider the endomorphism $W_\Phi$ of $\C O(SU_q(N)) \ot \C A$ defined by the formula
\[
W_\Phi(x \ot y) := (1 \ot \Phi)\De(x) \cd (1 \ot y) .
\]
It can be verified that $W_\Phi$ is a linear automorphism of the algebraic tensor product $\C O(SU_q(N)) \ot \C A$ and that this linear automorphism induces a unitary operator $W_\Phi$ on the Hilbert space tensor product $L^2(SU_q(N)) \hot L^2(A,\mu)$.
%

For each $k \in \{0,1,\ldots,\ell\}$ we define the linear automorphism $W_\Phi^k$ of the algebraic tensor product $\C O(SU_q(N)) \ot \La^k(\cc^\ell) \ot \C A$ by putting
\[
W^k_\Phi(x \ot \om \ot y) := (W_\Phi)_{13}(x \ot \om \ot y) = x_{(1)} \ot \om \ot \Phi(x_{(2)}) y .
\]
It clearly holds that $W^k_\Phi$ induces a unitary operator $W^k_\Phi$ on the Hilbert space tensor product $L^2(SU_q(N)) \hot \La^k(\cc^\ell) \hot L^2(A,\mu)$.

In the statement of the next lemma we apply the notation $D_q \hot 1$ for the closure of the unbounded operator
\[
\C D_q \ot 1 = (\ov \pa + \ov \pa^\da) \ot 1 : \Om_M \ot \C A \to L^2(\Om_M,h) \hot L^2(A,\mu) .
\]
Since the unbounded operator $\ov{\pa} + \ov{\pa}^\da$ is essentially selfadjoint (by Theorem \ref{t:dada}) it holds that $D_q \hot 1$ is selfadjoint.
%

\begin{lemma}\label{l:univee}
  The linear automorphism $W^k_\Phi$ of $\C O(SU_q(N)) \ot \La^k(\cc^\ell) \ot \C A$ restricts to a linear automorphism $(W_M^k)_\Phi$ of the subspace $\Om_M^k \ot \C A$ for all $k \in \{0,1,\ldots,\ell\}$. It moreover holds that the commutator
  \begin{equation}\label{eq:dirmuluni}
\big[ \op_{k = 0}^\ell (W_M^k)_\Phi, (\ov \pa + \ov \pa^\da) \ot 1 \big] : \Om_M \ot \C A \to \Om_M \ot \C A 
\end{equation}
is equal to zero. In particular, we get that $\op_{k = 0}^\ell (W_M^k)_\Phi$ induces a unitary operator $(W_M)_\Phi$ on $L^2(\Om_M,h) \hot L^2(A,\mu)$ satisfying that $(W_M)_\Phi (D_q \hot 1) = (D_q \hot 1) (W_M)_\Phi$. 
\end{lemma}
\begin{proof}
  Let $\eta \in \C U_q(\G{su}(N))$ be given and remark that
  \[
  \begin{split}
    (d_\eta \ot 1) W_\Phi(x \ot y) & = (d_\eta \ot \Phi)( \De(x) ) \cd (1 \ot y)
    = (1 \ot \Phi)\De( d_\eta(x)) \cd (1 \ot y) \\
    & = W_\Phi(d_\eta \ot 1)(x \ot y) 
    \end{split}
  \]
  for all $x \in \C O(SU_q(N))$ and $y \in \C A$. This shows that $W^k_\Phi(d_\eta \ot T \ot 1) = (d_\eta \ot T \ot 1)W^m_\Phi$ for every $k,m \in \{0,1,\ldots,\ell\}$ and every linear operator $T : \La^m(\cc^\ell) \to \La^k(\cc^\ell)$. This in turn implies that both $W^k_\Phi$ and $(W^k_\Phi)^{-1}$ preserve the subspace $\Om_M^k \ot \C A$ and moreover that the commutator in \eqref{eq:dirmuluni} is trivial. The remaining claims of the lemma now follow by taking closures. 
\end{proof}

Remark that we may represent the minimal tensor product $C(\B CP_q^\ell) \ot_{\min} A$ faithfully on the Hilbert space tensor product $L^2(\Om_M,h) \hot L^2(A,\mu)$. It should then be emphasized that the unitary operator $(W_M)_\Phi$ implements the unital $*$-homomorphism $(1 \ot \Phi)\de : C(\B CP_q^\ell) \to C(\B CP_q^\ell) \ot_{\min} A$, where $\de$ denotes the right coaction of $C(SU_q(N))$ on $C(\B CP_q^\ell)$. In other words, we have that
\begin{equation}\label{eq:coprod}
(1 \ot \Phi)\de(x) = (W_M)_\Phi (x \ot 1) (W_M)_\Phi^* \quad \T{for all } x \in C(\B CP_q^\ell) ,
\end{equation}
where both sides of the equation are understood to be bounded operators on $L^2(\Om_M,h) \hot L^2(A,\mu)$.

For a single vector $\xi \in L^2(A,\mu)$ we introduce the bounded operator $T_\xi : L^2(\Om_M,h) \to L^2(\Om_M,h) \hot L^2(A,\mu)$ given by the formula $T_\xi(\ze) := \ze \ot \xi$. We record that the operator norm of $T_\xi$ is equal to the Hilbert space norm of $\xi$. The bounded operator $T_\xi$ intertwines the unbounded selfadjoint operators $D_q$ and $D_q \hot 1$ in so far that we have the inclusion $T_\xi D_q \su (D_q \hot 1) T_\xi$ of unbounded operators. Thus, by taking adjoints we also obtain the inclusion $T_\xi^*(D_q \hot 1) \su D_q T_\xi^*$.  

For two vectors $\xi,\eta \in L^2(A,\mu)$ we define the linear functional $\varphi_{\xi,\eta} : A \to \cc$ by putting $\varphi_{\xi,\eta}(x) := \inn{\xi, \pi(x) \eta}$ and this linear functional can be promoted to a slice map $1 \ot \varphi_{\xi,\eta}$ on the minimal tensor product $C(\B CP_q^\ell) \ot_{\min} A$. Upon recalling that our minimal tensor product is represented faithfully on the Hilbert space tensor product $L^2(\Om_M,h) \hot L^2(A,\mu)$ we then obtain the relationship
\[
(1 \ot \varphi_{\xi,\eta})(z) = T_\xi^* z T_\eta \quad \T{for all } z \in C(\B CP_q^\ell) \ot_{\T{min}} A ,
\]
where it is understood that both sides are operators on the Hilbert space $L^2(\Om_M,h)$.

\begin{lemma}\label{l:estfunc}
  Let $\xi,\eta \in \C A$ and let $x\in \Lip_{D_q}(\B CP_q^\ell)$. It holds that
$(1 \ot \varphi_{\xi,\eta} \Phi)\de(x)$ belongs to the Lipschitz algebra $\Lip_{D_q}(\B CP_q^\ell)$ and we have the estimate
\[
L_{D_q}^{\max}\big( (1 \ot \varphi_{\xi,\eta} \Phi) \de(x) \big)\leq \| \xi \|_2 \cd \| \eta \|_2 \cd L_{D_q}^{\max}(x). 
\]
\end{lemma}
\begin{proof}
  It suffices to show that $(1 \ot \varphi_{\xi,\eta} \Phi) \de(x)$ maps $\Om_M$ into the domain of $D_q$ and that
  \[
\big[D_q, (1 \ot \varphi_{\xi,\eta} \Phi) \de(x) \big](\ze) = T_\xi^* (W_M)_\Phi (d_{\max}(x) \ot 1) (W_M)_\Phi^* T_\eta(\ze)
\]
for all $\ze \in \Om_M$. But this follows in a straightforward fashion from the equivariance properties combined in Lemma \ref{l:univee} and \eqref{eq:coprod} together with the observations made before the statement of the present lemma. To wit, we use that $(1 \ot \varphi_{\xi,\eta}\Phi) \de(x) = T_\xi^* (W_M)_\Phi (x \ot 1) (W_M)_\Phi^* T_\eta$. 
\end{proof}

\subsection{Relationship with the D\k{a}browski-Sitarz spectral triple}\label{ss:dabsit}
Let us spend some time clarifying the relationship between the even unital spectral triple $\big( \C O(\B CP_q^1), L^2(\Om_1,h), D_q\big)$ treated earlier in this section and the D\k{a}browski-Sitarz spectral triple over the Podle\'s sphere, see \cite{DaSi:DSP}. We are here focusing on the presentation from \cite{NeTu:LFQ} since this picture is compatible with the investigations of the spectral metric properties of the Podle\'s sphere given in \cite{AgKa:PSM,AgKaKy:PSC}.

We apply the notation $u = \ma{cc}{a^* & -qb \\ b^* & a}$ for the entries of the fundamental unitary for quantum $SU(2)$. The standard \emph{Podle\'s sphere} $C(S_q^2)$ is defined as the smallest unital $C^*$-subalgebra of $C(SU_q(2))$ containing $A := bb^*$ and $B := ab^*$. The corresponding coordinate algebra is denoted by $\C O(S_q^2)$ (and is generated as a unital $*$-subalgebra by $A$ and $B$).

The dual pairing described in Theorem \ref{t:dualpair} gives rise to an alternative left action of $\C U_q(\G{su}(2))$ on $\C O(SU_q(2))$ defined by putting $\pa_\eta(x) := (1 \ot \inn{\eta, \cd} )\De(x)$. This left action provides the coordinate algebra $\C O(SU_q(2))$ with the structure of a $\zz$-graded algebra with homogeneous subspaces defined by
\[
\C A_n := \big\{ x \in \C O(SU_q(2)) \mid \pa_K(x) = q^{-n/2} x \big\} \quad n \in \B Z .
\]
It holds that $\C A_0 = \C O(S_q^2)$ so that $\C A_n$ becomes a left module over $\C A_0$ for every $n \in \B Z$. We let $H_+$ and $H_-$ denote the Hilbert spaces obtained by taking the closure of respectively $\C A_1$ and $\C A_{-1}$ inside the Hilbert space $L^2(SU_q(2))$. The left action of $\C A_0$ on the dense subspaces $\C A_1$ and $\C A_{-1}$ can then be promoted to a pair of injective unital $*$-homomorphisms $\rho_{\pm} : C(S_q^2) \to \B B(H_{\pm})$. Taking the direct sum of these two representations we obtain a faithful representation $\rho$ of the Podle\'s sphere on the $\zz/2\zz$-graded Hilbert space $H_+ \op H_-$.

The linear operators $\pa_E : \C A_1 \to \C A_{-1}$ and $\pa_F : \C A_{-1} \to \C A_1$ can be promoted and combined into an odd unbounded operator
\[
\ma{cc}{0 & \pa_F \\ \pa_E & 0} : \C A_1 \op \C A_{-1} \to H_+ \op H_- 
\]
which turns out to be essentially selfadjoint. We apply the notation $\dirac_q$ for the corresponding selfadjoint closure.

The \emph{D\k{a}browski-Sitarz spectral triple}, in the presentation of Neshveyev and Tuset, agrees with the even unital spectral triple given by $\big( \C O(S_q^2), H_+ \op H_-, \dirac_q\big)$. As discussed after Definition \ref{d:specmet}, the D\k{a}browski-Sitarz spectral triple also exists in a Lipschitz version which we denote by $\big( \Lip_{\dirac_q}(S_q^2),H_+ \op H_-, \dirac_q \big)$. 

On the other hand, following the constructions of the present paper, we have that $\C O(\B CP_q^1)$ is generated by $A = bb^*$ and $C:= ab$ as a unital $*$-subalgebra of $\C O(SU_q(2))$ whereas $\Om_1 = \Om_1^0 \op \Om_1^1$ has homogeneous components given by
\[
\begin{split}
\Om_1^0 & = \big\{ x \in \C O(SU_q(2)) \mid d_K(x) = q^{1/2} x \big\} \quad \T{ and } \\
\Om_1^1 & = \big\{ x \in \C O(SU_q(2)) \mid d_K(x) = q^{-1/2} x \big\} .
\end{split}
\]
Furthermore, the relevant essentially selfadjoint unbounded operator is given by
\[
\C D_q = \ov{\pa} + \ov{\pa}^\da = q^{-1/2} \ma{cc}{0 & d_E \\ d_F & 0} : \Om_1^0 \op \Om_1^1 \to L^2( \Om_1^0,h) \op L^2(\Om_1^1,h) .
\]

We now introduce the key ingredient needed for describing the relationship between the two even unital spectral triples reviewed in this subsection. Define the $*$-automorphism $T : C(SU_q(2)) \to C(SU_q(2))$ by putting
\[
T(u_{ij}) := q^{j-i} u_{ji} \quad i,j \in \inn{2}
\]
and record that $T^2$ agrees with the identity operator. We refer to $T$ as the \emph{transpose} and the first result regarding this operation is that it leaves the Haar state invariant.

\begin{lemma}\label{l:transposeI}
We have the identity $hT = h$.
\end{lemma}
\begin{proof}
  This can be viewed as a consequence of the explicit formula for the Haar state on quantum $SU(2)$ obtained by Woronowicz in \cite[Appendix A]{Wor:CMP}. We recall in this respect that 
  \[
h\big(a^k b^i (b^*)^j\big) = \de_{ij} \de_{k0} \cd \frac{1 - q^2}{1 - q^{2(1 + j)}} \quad \T{for all } k \in \zz \, , \, \, i,j \in \nn_0,
\]
where we use the convention that $a^k := (a^*)^{-k}$ for $k < 0$.
\end{proof}

As a consequence of Lemma \ref{l:transposeI}, we get that the transpose $T$ induces a selfadjoint unitary operator $T : L^2(SU_q(2)) \to L^2(SU_q(2))$. 

The second result regarding the transpose investigates its intertwining properties with respect to the two left actions $d$ and $\pa$ of $\C U_q(\G{su}(2))$ on the coordinate algebra $\C O(SU_q(2))$. Notice in this respect that $T$ restricts to a $*$-automorphism of the coordinate algebra $\C O(SU_q(2))$ and introduce the $*$-automorphism $\nu : \C U_q(\G{su}(2)) \to \C U_q(\G{su}(2))$ by putting
\[
\nu(K) := K^{-1} \, \, \T{ and } \, \, \, \nu(E) := - F .
\]

\begin{lemma}\label{l:transposeII}
For every $\eta \in \C U_q(\G{su}(2))$ it holds that $d_\eta T(x) = T \pa_{\nu(\eta)}(x)$ for all $x \in \C O(SU_q(2))$.
\end{lemma}
\begin{proof}
The first step is to verify the relevant identity in the case where $\eta \in \{E,F,K\}$ and $x = u_{ij}$ for some $i,j \in \inn{2}$ (this can be done by a straightforward computation). The more general case where $x$ is arbitrary (and $\eta \in \{E,F,K\}$) then follows by noting that $T d_\eta T$ obeys the same algebraic rules as $\pa_{\nu(\eta)}$ with respect to linear combinations and products in the coordinate algebra $\C O(SU_q(2))$. The case where both $\eta$ and $x$ are arbitrary finally follows since both sides of the equation are well-behaved with respect to products and sums of elements in $\C U_q(\G{su}(2))$.
\end{proof}

An application of Lemma \ref{l:transposeI} in combination with Lemma \ref{l:transposeII} shows that the transpose restricts to two selfadjoint unitary operators $T_+ : H_+ \to L^2(\Om_1^0,h)$ and $T_- : H_- \to L^2(\Om_1^1,h)$ which intertwine the relevant unbounded selfadjoint operators (up to a factor $-q^{-1/2}$) via the identity
\[
D_q (T_+ \op T_-) = -q^{-1/2} (T_+ \op T_-) \dirac_q .
\]
It is moreover clear from the definition of the transpose that $T$ restricts to a $*$-isomorphism $T : C(S_q^2) \to C(\B CP_q^1)$.

These observations yield the main result of this subsection, providing a precise relationship between the slip-norms associated to our two even unital spectral triples.

\begin{prop}\label{p:dabsitdada}
  The $*$-isomorphism $T : C(S_q^2) \to C(\B CP_q^1)$ restricts to a $*$-isomorphism between the Lipschitz algebras $T : \Lip_{\dirac_q}(S_q^2) \to \Lip_{D_q}(\B CP_q^1)$ satisfying the identity
  \[
L_{D_q}^{\max}\big(T(x)\big) = q^{-1/2} L_{\dirac_q}^{\max}(x) \quad \mbox{for all } x \in \Lip_{\dirac_q}(S_q^2).
\]
In particular, it holds that $\big( \Lip_{D_q}(\B CP_q^1), L^2(\Om_1,h), D_q \big)$ is a spectral metric space if and only if $\big( \Lip_{\dirac_q}(S_q^2), H_+ \op H_-, \dirac_q\big)$ is a spectral metric space.
\end{prop}

\section{The extension theorem}
Throughout this whole section we let $M \in \zz$ be fixed since this $M$ appears in the definition of the even unital spectral triple $\big( \C O( \B CP_q^\ell),L^2(\Om_M,h), D_q \big)$ from Theorem \ref{t:dada}. 

Let us introduce the norm-dense unital $*$-subalgebra $\Lip(I_q) \su C(I_q)$ of \emph{$q$-Lipschitz functions} on the quantized interval $I_q = \Sp(y_\ell)$ by putting
\[
\Lip_{D_q}(I_q) := C(I_q) \cap \Lip_{D_q}(\B CP_q^\ell) . 
\]
Notice here that we are identifying $C(I_q)$ with the unital $C^*$-subalgebra of quantum projective space generated by $y_\ell$, see the discussion in Subsection \ref{ss:quasph}. The $q$-Lipschitz functions are equipped with the slip-norm $L_{D_q}^{\max} : \Lip_{D_q}(I_q) \to [0,\infty)$ defined as the restriction of the slip-norm $L_{D_q}^{\max} : \Lip_{D_q}(\B C P_q^\ell) \to [0,\infty)$. Recall here that both $L_{D_q}^{\max}$ and the Lipschitz algebra $\Lip_{D_q}(\B CP_q^\ell)$ come from the even unital spectral triple on quantum projective space. 

Our aim is now to prove an extension theorem saying that the pair $\big( \Lip_{D_q}(\B CP_q^\ell), L_{D_q}^{\max} \big)$ is a compact quantum metric space if and only if the pair $\big( \Lip_{D_q}(I_q),L_{D_q}^{\max}\big)$ is a compact quantum metric space.

The strategy for proving this extension theorem is inspired from the constructions appearing in the paper \cite{AgKaKy:PSC}, regarding quantum Gromov-Hausdorff convergence of the standard Podle\'s spheres with respect to the D\k{a}browski-Sitarz spectral triple, \cite{DaSi:DSP}. In the paper \cite{AgKaKy:PSC}, the authors construct a sequence of positive unital endomorphisms of the standard Podle\'s sphere $C(S_q^2) = C(\B CP_q^1)$ such that all the endomorphisms appearing in this sequence have finite dimensional images. Moreover, the whole sequence converges pointwise to the identity operator in a way which can be estimated by the Monge-Kantorovich metric on the state space of the standard Podle\'s sphere.

We shall see that similar constructions for the quantum projective space $C(\B CP_q^\ell)$ will help us verify the conditions in Theorem \ref{t:approx} for the pair $\big(\Lip_{D_q}(\B CP_q^\ell),L_{D_q}^{\max}\big)$ under the assumption that $\big(\Lip_{D_q}(I_q),L_{D_q}^{\max}\big)$ is a compact quantum metric space. More precisely, using the techniques from \cite{AgKaKy:PSC} we are able to establish the finite diameter condition and prove the existence of arbitrarily precise $\ep$-approximations.

\subsection{Weak-$*$ convergence to the counit}\label{ss:weakstar}
In this subsection, we shall apply the Haar state $h : C(SU_q(N))\to \cc$ to construct a sequence of states $\{h_k\}_{k = 0}^\infty$ on $C(\B CP_q^\ell)$ which converges in the weak-$*$ topology to the restriction of the counit $\epsilon : C(\B CP_q^\ell) \to \cc$. Notice that a related result already appears in \cite[Proposition 2.4 and Theorem 2.5]{IzNeTu:PBD}, but we are currently unaware of the precise relationship between the statements in \cite{IzNeTu:PBD} and our Proposition \ref{p:weak*}. 

For each $k\in \mathbb{N}_0$ we put $a_k := (z_N^*)^k z_N^k$ and introduce the state 
\[
h_k :  C(\B CP_q^\ell) \to \B C \quad
h_k(x) :=   h\big((z_{N}^*)^k x z_{N}^k\big) \cdot h(a_k)^{-1} .
\]
Notice that $h_0$ agrees with the restriction of the Haar state to quantum projective space. 

In order to prove that the sequence $\{ h_k\}_{k = 0}^\infty$ converges to the restricted counit we start out by reducing this problem to the quantized interval $I_q$. This is carried out by means of a conditional expectation of $C(\B CP_q^\ell)$ onto $C(I_q)$. This conditional expectation is also applied in \cite{Nag:HQG} (and in many other places) but for the sake of completeness we review the construction here.

For $N \geq 3$ we define the surjective unital $*$-homomorphism $\Phi : C(SU_q(N)) \to C(SU_q(N - 1))$ by putting
\begin{equation}\label{eq:defphi}
\Phi(u_{ij}) := \fork{ccc}{ u_{ij} & \T{for} & i,j < N \\ \de_{ij} & \T{for} & i = N \T{ or } j = N} . 
\end{equation}
For $N = 2$ we consider the unital $C^*$-algebra of continuous functions on the unit circle $C(\B T)$ and let $w : \B T \to \cc$ denote the inclusion. We then define the surjective unital $*$-homomorphism $\Phi : C(SU_q(2)) \to C(\B T)$ by the formula
\begin{equation}\label{eq:defphiII}
\Phi(u_{ij}) := \fork{ccc}{ w & \T{for} & i = j = 1 \\ \de_{ij} w^{-1} & \T{for} & i = 2 \T{ or } j = 2} .
\end{equation}
In both cases, we record that $\Phi$ is compatible with the coalgebra structures so that $(\Phi \ot \Phi)\De = \De \Phi$ (where the coproduct on $C(\B T)$ is induced by the group structure of the compact group $\B T$). We therefore obtain a conditional expectation $E : C(\B CP_q^\ell) \to C(\B CP_q^\ell)$ by putting
\[
E(x) := (1 \ot h \Phi)\de(x),
\]
where $h$ denotes the Haar state both for $N \geq 3$ and for $N = 2$ (for $N = 2$ the Haar state is just given by Riemann integration with respect to the Haar measure on the unit circle). We remark that the operation $(h \Phi \ot 1) \de$ agrees with the identity map on $C(\B CP_q^\ell)$ (in fact this even holds on $C(S_q^{2\ell + 1})$ for $N \geq 3$), but in the definition of $E$ we are slicing with the state $h \Phi : C(SU_q(N)) \to \B C$ on the \emph{right} leg of the tensor product.


\begin{prop}
The image of the conditional expectation $E : C(\B CP_q^\ell) \to C(\B CP_q^\ell)$ agrees with the unital $C^*$-subalgebra $C(I_q)$.
\end{prop}
\begin{proof}
  For $N = 2$ we consider the strongly continuous circle action $\si$ on $C(\B CP_q^1)$ satisfying that $\si_\la(z_2 z_2^*) = z_2 z_2^*$ and $\si_\la(z_1 z_2^*) = \la^2 z_1 z_2^*$ for all $\la \in \B T$. The statement of the proposition can then be obtained by noting that the image of $E$ agrees with the fixed point algebra for the circle action $\si$. 

For $N \geq 3$, the result is a consequence of \cite[Theorem 3.3]{Nag:HQG}. Indeed, we may extend the conditional expectation $E$ to a conditional expectation on the quantum sphere, $\wit{E} : C(S_q^{2\ell + 1}) \to C(S_q^{2\ell + 1})$ defined by the same formula $\wit{E}(x) = (1 \ot h \Phi)\de(x)$ (but now using the right coaction of quantum $SU(N)$ on the quantum sphere). As a consequence of \cite[Theorem 3.3]{Nag:HQG} we then get that the image of $\wit{E}$ coincides with $C^*(z_N,1)$ (the smallest unital $C^*$-subalgebra of $C(S_q^{2\ell + 1})$ which contains $z_N$). The result of the proposition now follows by noting that $C^*(z_N,1) \cap C(\B CP_q^\ell) = C(I_q)$.
\end{proof}

The main point is that the state $h_k : C(\B CP_q^\ell)\to \cc$ for $k \in \nn_0$ and the restricted counit $\epsilon : C(\B CP_q^\ell)\to \cc$ are invariant under the conditional expectation $E$:

\begin{lemma}\label{l:inv}
Let $x\in C(\B CP_q^\ell)$ and $k\in \B N_0$. It holds that
\[
h_k(E(x)) = h_k(x ) \, \, \mbox{ and } \, \, \, \epsilon(E(x)) = \epsilon(x). 
\]
\end{lemma}
  \begin{proof}
    We first notice that the bi-invariance property of the Haar state \eqref{eq:haar} implies $h E(x) = h \Phi\big( h(x) \cd 1\big) = h(x)$ establishing the relevant identity for $k = 0$. For $k \in \nn$, the most efficient approach is to apply the modular properties of the Haar state. By the discussions near \eqref{eq:modular} we get that $\te(z_N^*) = q^{2(1 - N)} z_N^*$ and hence that $z_N^k \te(z_N^*)^k = z_N^k (z_N^*)^k q^{2k(1-N)}$ belongs to $C(I_q)$. Since $E$ is a conditional expectation onto $C(I_q)$ and $h E = h$ we obtain that
    \[
    \begin{split}
    h_k(E(x)) & = h(a_k)^{-1} \cd h\big( (z_N^*)^k E(x) z_N^k\big) 
    = h(a_k)^{-1} \cd h\big( E(x) z_N^k \te(z_N^*)^k \big) \\
    & =  h(a_k)^{-1} \cd h\big( x z_N^k \te(z_N^*)^k \big) 
    = h_k(x) .
    \end{split}
    \]
    The final claim regarding the counit follows by checking that $\Phi(z_i z_j^*) = \de_{iN} \de_{jN} = \epsilon(z_i z_j^*) \cd 1$ for all $i,j \in \inn{N}$ so that we also get the identity $\Phi(x) = \epsilon(x) 1$. Indeed, this latter identity entails that $\epsilon( E(x) ) = h \Phi(x) = \epsilon(x)$.
\end{proof}

After these preparations we are ready to prove that the sequence of states $\{h_k\}_{k = 0}^\infty$ satisfies the desired convergence property.

\begin{prop}\label{p:weak*}
The sequence of states $\{h_k\}_{k = 0}^\infty$ on $C(\B CP_q^\ell)$ converges to the restricted counit $\epsilon : C(\B CP_q^\ell)\to \B C$ in the weak-$*$ topology. 
\end{prop}
\begin{proof}
By Lemma \ref{l:inv} it is sufficient to show pointwise convergence on $C(I_q)$. Moreover, since $ \{ y_\ell^m  \mid m \in \B N_0 \} $ is linearly norm-dense in $C(I_q)$, we only need to show that $\lim_{k\to \infty} h_k(y_\ell^m)=\epsilon(y_\ell^m)$ for every $m \in \B N$ (so $m \neq 0$). For such $m$ it holds that $\epsilon(y_\ell^m)=0 $, and moreover for every $k \in \nn_0$ we have the estimate
\[
\begin{split}
  h_k(y_\ell^m) &= h\big( (z_N^*)^k y_\ell^m z_N^k \big)\cd h(a_k)^{-1} 
  = q^{2k m} \cd  h\big( (z_N^*)^k z_N^k y_\ell^m\big) \cd h(a_k)^{-1} \\
  & = q^{2k m} \cd  h\big( a_k^{1/2} y_\ell^m a_k^{1/2 }\big) \cd h(a_k)^{-1} \leq q^{2km} \cd \| y_\ell^m \|  , 
\end{split}
\]
where the third identity follows since both the positive element $a_k = (z_N^*)^k z_N^k$ and $y_\ell^m$ belong to $C(I_q)$ (which is a commutative $C^*$-algebra).
\end{proof}

\subsection{Monge-Kantorovich-convergence to the counit}
Let us for a little while assume that the pair $\big(\Lip_{D_q}(I_q),L_{D_q}^{\max}\big)$ is a compact quantum metric space (this will be established later on in Theorem \ref{thm:qms}). It then follows from Proposition \ref{p:weak*} that the restricted sequence of states $\big\{ h_k|_{C(I_q)}\big\}_{k = 0}^\infty$ converges to the restricted counit $\epsilon|_{C(I_q)}$ with respect to the Monge-Kantorovich distance $\mk_{L_{D_q}^{\max}}$ on the state space $S\big(C(I_q)\big)$.

In this subsection, we shall prove that the Monge-Kantorovich distance between $h_k$ and the restricted counit $\epsilon : C(\B CP_q^\ell) \to \B C$ actually agrees with the Monge-Kantorovich distance between $h_k|_{C(I_q)}$ and $\epsilon|_{C(I_q)}$. As a consequence we get that the sequence of states $\{ h_k\}_{k = 0}^\infty$ converges to the restricted counit $\epsilon : C(\B CP_q^\ell) \to \B C$ with respect to the Monge-Kantorovich distance on the state space $S\big(C(\B CP_q^\ell)\big)$.  

From now on we {\it no longer} assume that $\big(\Lip_{D_q}(I_q),L_{D_q}^{\max}\big)$ is a compact quantum metric space.

%

We first establish that the conditional expectation $E : C(\B C P_q^\ell)\to C(\B C P_q^\ell )$ is a contraction for the slip-norm $L_{D_q}^{\max} : \Lip_{D_q}( \B C P_q^\ell) \to [0,\infty)$.


\begin{lemma}\label{l:lipcon}
Let $x\in \Lip_{D_q}(\B CP_q^\ell)$. It holds that $E(x) \in \Lip_{D_q}(I_q)$ and we have the inequality
\[
L_{D_q}^{\max}( E(x)) \leq L_{D_q}^{\max}(x) . 
\]
\end{lemma}
\begin{proof}
We put $A = C(SU_q(\ell))$ for $N\geq 3$ and $A = C(\B T)$ for $N = 2$ and in both cases we equip $A$ with the Haar state $h$ (which is indeed faithful, see \cite[Theorem 1.1]{Nag:HQG} for the case where $N \geq 3$). We then have the surjective unital $*$-homomorphism $\Phi : C(SU_q(N)) \to A$ from \eqref{eq:defphi} and \eqref{eq:defphiII}. The result of the lemma therefore follows from Lemma \ref{l:estfunc} by noting that $E = (1 \ot \varphi_{1,1} \Phi)\de$.
\end{proof}

We may now prove the main result of this subsection. 

\begin{prop}\label{p:mkrest}
For all $k \in \nn_0$ we have the identity
\[
\mk_{L_{D_q}^{\max}}\big(h_k,\epsilon|_{C(\B CP_q^\ell)} \big)=\mk_{L_{D_q}^{\max}}\big(h_k|_{C(I_q)},\epsilon|_{C(I_q)}\big) . 
\]
\end{prop}
\begin{proof}
The inequality $\mk_{L_{D_q}^{\max}}\big(h_k|_{C(I_q)},\epsilon|_{C(I_q)}\big) \leq \mk_{L_{D_q}^{\max}}(h_k,\epsilon|_{C(\B CP_q^\ell)} )$ is satisfied since the $q$-Lipschitz functions $\Lip_{D_q}(I_q)$ are contained in the Lipschitz algebra $\Lip_{D_q}(\B CP_q^\ell)$ and the relevant slip-norms agree on the $q$-Lipschitz functions. To show the reverse inequality, we let $x\in \Lip_{D_q}(\B C P_q^\ell) $ with $L_{D_q}^{\max}(x) \leq 1$ be given. Applying Lemma \ref{l:inv} and Lemma \ref{l:lipcon}, we obtain the estimate
\[ 
\big|h_k(x) -\epsilon(x) \big| = \big|  h_k(E(x)) -\epsilon(E(x))\big| \leq \mk_{L_{D_q}^{\max}}(h_k|_{C(I_q)},\epsilon|_{C(I_q)}). 
\]
Hence, it follows that $\mk_{L_{D_q}^{\max}}\big(h_k,\epsilon|_{C(\B CP_q^\ell)}\big) \leq \mk_{L_{D_q}^{\max}}\big(h_k|_{C(I_q)},\epsilon|_{C(I_q)}\big)$. 
\end{proof}

\subsection{Finite dimensional approximations}
In this subsection we are going to prove the extension theorem. The main step is to show that if $\big( \Lip_{D_q}(I_q), L_{D_q}^{\max}\big)$ is a compact quantum metric space, then there exist arbitrarily precise finite dimensional approximations for the pair $\big( \Lip_{D_q}(\B CP_q^\ell), L_{D_q}^{\max}\big)$ (cf. Theorem \ref{t:approx}). For each $k\in \B N_0$ our prospective finite dimensional approximation is defined as the positive unital map
\[
\beta_k : C(\B CP_q^\ell ) \to C(SU_q(N)) \quad \beta_k(x) := ( h_k \otimes 1 ) \de(x),
\]
where we recall that $\de$ denotes the right coaction of quantum $SU(N)$ on quantum projective space, see Subsection \ref{ss:quasph}. Record that $\beta_0(x)=h(x)\cd 1$ since $h_0$ agrees with the restriction of the Haar state to $C(\B CP_q^\ell)$.

Our first result is that $\be_k$ has finite dimensional image. The proof relies on the Peter-Weyl decomposition of $\C O(SU_q(N))$ together with the Schur orthogonality relations, see Subsection \ref{ss:weylschur}.

\begin{lemma}\label{l:betafindim}
For all $k\in \nn_0$ it holds that $\beta_k : C(\B CP_q^\ell) \to C(SU_q(N))$ has finite dimensional image. 
\end{lemma}
\begin{proof}
  We are in fact going to prove a stronger result, namely that a certain extension of $\be_k$ has finite dimensional image. Recall from Subsection \ref{ss:weakstar} that $a_k := (z_N^*)^k z_N^k$ and define the state $H_k : C(SU_q(N)) \to \B C$ by putting $H_k(x) := h((z_N^*)^k x z_N^k ) \cd h(a_k)^{-1}$ for all $x \in C(SU_q(N))$. We may then introduce the positive unital map
  \[
B_k := (H_k \ot 1) \De : C(SU_q(N)) \to C(SU_q(N)) .
\]
Since the state $h_k$ agrees with the restriction of $H_k$ to $C(\B CP_q^\ell)$ we obtain that $B_k(x) = \be_k(x)$ for all $x \in C(\B CP_q^\ell)$. Hence, to establish the result of the lemma it suffices to show that $B_k$ has finite dimensional image.
%

Recall that $\theta : \C O(SU_q(N))\to \C O(SU_q(N))$ denotes the modular automorphism of the Haar state. Putting $w_k := z_N^k \te( (z_N^k)^*)$ we then obtain that
\[
H_k(x) =  h\big((z_N^*)^k x z_N^k \big) \cd h(a_k)^{-1}  = h\big(\te^{-1}(z_N^k)\cd (z_N^*)^k x \big) \cd h(a_k)^{-1}
=  h(w_k^* x) \cd h(a_k)^{-1}
\]
for all $x \in \C O(SU_q(N))$. Using the Peter-Weyl decomposition we may choose a finite subset $F \su \widehat{\C O(SU_q(N))}$ such that $w_k \in \sum_{\al \in F} \C C(\al)$. For every $\be \in \widehat{\C O(SU_q(N))}\sem F$ and $x \in \C C(\be)$ it holds that $\De(x) \in \C C(\be) \ot \C C(\be)$ and the Schur orthogonality relations then imply that
\[
B_k(x) = (H_k \ot 1)\De(x) = h(a_k)^{-1} \cd h(w_k^* x_{(1)}) \cd x_{(2)} = 0 .
\]
Applying the Peter-Weyl decomposition again we can now conclude that
\[
B_k\big( \C O(SU_q(N)) \big) = B_k\big( \sum_{\al \in \widehat{\C O(SU_q(N))}} \C C(\al) \big) = B_k\big( \sum_{\al \in F} \C C(\al) \big) .
\]
This shows that $B_k\big( \C O(SU_q(N))\big)$ is finite dimensional and by continuity and density we infer that the image of $B_k : C(SU_q(N)) \to C(SU_q(N))$ must be finite dimensional too.
\end{proof}

\begin{prop}\label{p:betaest}
Let $x\in \Lip_{D_q} (\B C P_q^\ell)$ and let $k\in \B N_0$. We have the inequality
\[
\| x-\beta_k(x) \| \leq \mk_{L_{D_q}^{\max}}\big(h_k,\epsilon|_{C(\B CP_q^\ell)}\big) \cdot L_{D_q}^{\max}(x) . 
\]
\end{prop}
\begin{proof}
Let $\xi, \eta \in \C O(SU_q(N))$ be two vectors with $\| \xi \|_2$ and $\| \eta \|_2$ dominated by one (inside the Hilbert space $L^2(SU_q(N))$). Since $x=(\epsilon|_{C(\B CP_q^\ell)} \otimes 1)\de(x)$, we have the identities
\[
\begin{split}
\varphi_{\xi,\eta}(x-\beta_k(x)) & = \varphi_{\xi,\eta} \big( (\epsilon|_{C(\B CP_q^\ell)}-h_k) \ot 1 \big) \de(x) \\
& =(\epsilon|_{C(\B CP_q^\ell)}-h_k)( 1 \ot \varphi_{\xi,\eta} ) \de(x).
\end{split}
\]
We may now apply Lemma \ref{l:estfunc} to the case where $\Phi$ is the identity operator on $C(SU_q(N))$ (and $\mu$ is the Haar state $h$). Indeed, an application of this lemma and the definition of the Monge-Kantorovich metric yield the estimates:
\[
\begin{split}
\big| \varphi_{\xi,\eta} (x-\beta_k(x)) \big|
& \leq  \mk_{L_{D_q}^{\max}}\big(h_k,\epsilon|_{C(\B CP_q^\ell)}\big) \cdot L_{D_q}^{\max}\big((1 \ot \varphi_{\xi,\eta})\de(x) \big) \\
& \leq \mk_{L_{D_q}^{\max}}\big(h_k,\epsilon|_{C(\B CP_q^\ell)}\big) \cdot L_{D_q}^{\max}(x) .
\end{split}
\]
The result of the proposition then follows since $C(SU_q(N))$ is represented faithfully on $L^2(SU_q(N))$, see the discussion towards the end of Subsection \ref{sec:suqn}. Indeed, for every $y \in C(SU_q(N))$ we get that 
\[
\| y \| = \sup\big\{ | \varphi_{\xi,\eta}(y) | \, \mid \xi,\eta \in \C O(SU_q(N))  \, , \, \, \| \xi \|_2, \| \eta \|_2 \leq 1 \big\} . \qedhere
\]
\end{proof}

We are now prepared for the proof of the extension theorem. 

\begin{theorem}\label{t:extlem}
It holds that $\big(\Lip_{D_q}(\B CP_q^\ell),L_{D_q}^{\max}\big)$ is a compact quantum metric space if and only if $\big(\Lip_{D_q}(I_q),L_{D_q}^{\max}\big)$ is a compact quantum metric space.  
\end{theorem}
\begin{proof}
Since $\Lip_{D_q}(I_q)$ is a unital $*$-subalgebra of $\Lip_{D_q}(\B CP_q^\ell)$ and the two slip-norms agree on $\Lip_{D_q}(I_q)$, it follows from Theorem \ref{t:approx} that if $\big(\Lip_{D_q}(\B CP_q^\ell),L_{D_q}^{\max}\big)$ is a compact quantum metric space, then this holds for $\big(\Lip_{D_q}(I_q),L_{D_q}^{\max}\big)$ as well. 
  
  Suppose that $\big( \Lip_{D_q}(I_q),L_{D_q}^{\max} \big)$ is a compact quantum metric space. Then we get from Proposition \ref{p:weak*} that $\lim_{k\to \infty} \mk_{L_{D_q}^{\max}}\big(h_k|_{C(I_q)},\epsilon|_{C(I_q)}\big)  = 0$ and hence by Proposition \ref{p:mkrest} that also $\lim_{k\to \infty} \mk_{L_{D_q}^{\max}}\big(h_k,\epsilon|_{C(\B CP_q^\ell)}\big)=0$. 

  Let $\ep > 0$ be given. Choose a $k_0 \in  \B N_0$ such that $\mk_{L_{D_q}^{\max}}(h_{k_0},\epsilon|_{C(\B CP_q^\ell)}) \leq \ep$. If we let $i : C(\B CP_q^\ell)\to C(SU_q(N))$ denote the inclusion, it then follows from Lemma \ref{l:betafindim} and Proposition \ref{p:betaest} that $\big(i,\beta_{k_0}\big)$ is an $\ep$-approximation of $\big(\Lip_{D_q}(\B CP_q^\ell),L_{D_q}^{\max}\big)$.
  
  By Theorem \ref{t:approx}, it now only remains to show that $\big(\Lip_{D_q}(\B CP_q^\ell),L_{D_q}^{\max}\big)$ has finite diameter. To this end, we first notice that $\mk_{L_{D_q}^{\max}}(h_0,\epsilon|_{C(\B CP_q^\ell)}) <\infty$. This is a consequence of Proposition \ref{p:mkrest} together with our assumption that $\big( \Lip_{D_q}(I_q),L_{D_q}^{\max} \big)$ is a compact quantum metric space. The finite diameter condition for $\big(\Lip_{D_q}(\B CP_q^\ell),L_{D_q}^{\max}\big)$ then follows from Proposition \ref{p:betaest}. Indeed, we get the estimates
\[
\big\| [x] \big\|_{C(\B CP_q^\ell)/\cc 1} \leq \| x-h_0(x)\cd 1 \| = \| x-\beta_0(x) \|
\leq \mk_{L_{D_q}^{\max}}(h_0,\epsilon|_{C(\B CP_q^\ell)}) \cd L_{D_q}^{\max}(x) ,
\]
for all $x \in \Lip_{D_q}(\B CP_q^\ell)$.
\end{proof}

\section{The quantum metric structure of the quantized interval}
Let us again fix an integer $M \in \zz$ and recall that $M$ is relevant in the construction of the unital spectral triple $\big( \Lip_{D_q}(\B CP_q^\ell), L^2(\Om_M,h), D_q\big)$. Remark that this unital spectral triple restricts to a unital spectral triple over the $q$-Lipschitz functions namely $\big( \Lip_{D_q}(I_q), L^2(\Om_M,h), D_q\big)$.

In view of Theorem \ref{t:extlem}, our aim is now to prove that the pair $\big(\Lip_{D_q}(I_q),L_{D_q}^{\max} \big)$ is a compact quantum metric space (so that the unital spectral triple over the $q$-Lipschitz functions becomes a spectral metric space). As a first step in this direction, we compute derivatives of the indicator functions associated to isolated points in $I_q$, see \eqref{eq:mproj}.
%

\subsection{Derivatives of indicator functions}\label{ss:derind}
For each $k \in \{0,1,\ldots,\ell\}$ we recall from Section \ref{s:specgeo} that $L^2(\Om_M^k,h)$ denotes the closure of the twisted antiholomorphic forms $\Om_M^k$ inside the Hilbert space tensor product $L^2(SU_q(N)) \ot \La^k(\B C^\ell)$. 

For each $x \in \Lip_{D_q}(\B CP_q^\ell)$ we define the bounded operator $\Na(x) : L^2(\Om_M^0,h) \to L^2(\Om_M^1,h)$ by putting
\[
\Na(x)(\xi) := -d(x)(\xi) \quad \T{for all } \xi \in L^2(\Om_M^0,h).
\]
It then holds that $\Na : \Lip_{D_q}(\B CP_q^\ell) \to \B B\big(L^2(\Om_M^0,h),L^2(\Om_M^1,h)\big)$ is a closable derivation and we apply the notation $\ov{\Na}$ for the closure. When applying $\Na$ to a $q$-Lipschitz function $f$ we sometimes refer to the resulting operator as the \emph{gradient} of $f$.

%

In this subsection, we shall prove that the projection $p_m \in C(I_q)$ belongs to the domain of $\ov{\Na}$ for all $m \in \nn_0$ and compute the derivative $\ov{\Na}(p_m)$. We recall here that the projection $p_m$ was introduced in Subsection \ref{ss:quasph} and corresponds to the isolated point $q^{2m}$ in the spectrum of $y_\ell$.

For each $i \in \inn{\ell}$ we start out by defining the operation
\[
\Na_i := (-q)^{i - N} d_{F_i \clc F_\ell} : \C O(S_q^{2\ell +1}) \to \C O(SU_q(N)) .
\]
It is then relevant to notice that $\Na_i$ is a twisted derivation in the sense that
\[
\Na_i(x y) = \Na_i(x) d_{K_\ell^{-1}}(y) + d_{K_\ell}(x) \Na_i(y)
\]
for all $x,y \in \C O(S_q^{2\ell + 1})$, see \eqref{eq:homsph} and \eqref{eq:twileib}. Moreover, using \eqref{eq:calcrule} we may specify how $\Na_i$ behaves with respect to the adjoint operation:
\[
\Na_i(x^*) = (-q)^{i - N}d_{F_i \clc F_\ell}(x^*) = (-q)^{i - N}d_{S( E_\ell \clc E_i)}(x)^* = d_{E_i \clc E_\ell}(x)^*  .
\]
For every $r \in \inn{N}$ we also record the following formulae which are consequences of \eqref{eq:actgener}:
\begin{equation}\label{eq:nabvan}
\Na_i(z_r) = 0 \quad \T{ and } \quad \Na_i(z_r^*) = d_{E_i \clc E_\ell}(u_{Nr})^* = (-q)^{i-N} u_{ir}^* .
\end{equation}

By Lemma \ref{l:partial} the relationship between the derivation $\Na$ and the above twisted derivations can be clarified. Indeed, we have that
\begin{equation}\label{eq:nabcoord}
\Na(x)(\xi) = \sum_{i = 1}^\ell \Na_i(x)(\xi) \ot e_i \quad \T{for all } \xi \in L^2(\Om_M^0,h) \, \, \T{ and } \, \, \, x \in \C O(\B CP_q^\ell) .
\end{equation}

In the next three short lemmas we present a couple of commutation relations involving the twisted derivation $\Na_i$ and the generators for the coordinate algebra of the quantum sphere. We recall that $x_s := z_s z_s^*$ for all $s \in \inn{N}$ and that $y_r := \sum_{s = 1}^r x_s$ for all $r \in \inn{N}$.

\begin{lemma}\label{l:znab}
  For $i \in \inn{\ell}$ and $s,r \in \inn{N}$ with $s \neq r$ we have the identity
  \[
  z_s \Na_i(z_r^*) = \Na_i(z_r^*) z_s .
\]
\end{lemma}
\begin{proof}
This follows by using that $\Na_i$ is a twisted derivation together with \eqref{eq:actgener}, \eqref{eq:nabvan} and \eqref{eq:defsphere}:
  \[
  \begin{split}
    q^{-1/2} z_s \Na_i(z_r^*)
    & = d_{K_\ell}(z_s) \Na_i(z_r^*) = \Na_i(z_s z_r^*) = q^{-1} \Na_i(z_r^* z_s) \\
    & = q^{-1} \Na_i(z_r^*) d_{K_\ell^{-1}}(z_s) = q^{-1/2} \Na_i(z_r^*) z_s . \qedhere
  \end{split}
  \]
\end{proof}


\begin{lemma}\label{l:adjnab}
  For $i \in \inn{\ell}$ and $s,r \in \inn{N}$ with $s < r$ we have the identities
  \[
z_s^* \Na_i(z_r^*) = \Na_i(z_r^*) z_s^* \, \, \mbox{ and } \, \, \, 
x_s \Na_i(x_r) = q^2 \Na_i(x_r) x_s .
  \]
\end{lemma}
\begin{proof}
  The first identity follows from \eqref{eq:nabvan} and the defining relations for $\C O(SU_q(N))$. The second identity follows by applying \eqref{eq:defsphere} in combination with the first identity and Lemma \ref{l:znab}:
\[
  \begin{split}
  x_s \Na_i(x_r) & = q^{-1/2} x_s z_r \Na_i(z_r^*) = q^{-1/2} q^2 z_r x_s \Na_i(z_r^*) \\
  & = q^{-1/2} q^2 z_r \Na_i(z_r^*) x_s = q^2 \Na_i(x_r) x_s . \qedhere
  \end{split}
  \]
\end{proof}

\begin{lemma}\label{l:xpasty}
  For $i \in \inn{\ell}$ and $s,r \in \inn{N}$ with $s \leq r$ we have that
  \[
  x_s \Na_i( y_r) = q^2 \Na_i(y_r) x_s \, \, \mbox{ and } \, \, \,
y_s \Na_i(y_r) = q^2 \Na_i(y_r) y_s .  
  \]
\end{lemma}
\begin{proof}
  The second identity follows immediately from the first one so we focus on the first identity. The case where $r = N$ is trivially satisfied since $y_N = 1$ and $\Na_i(1) = 0$. The remaining case where $r \leq \ell$ follows from Lemma \ref{l:adjnab} since
  \[
\Na_i(y_r) = \Na_i(1 - \sum_{j = r + 1}^N x_j) = - \sum_{j = r + 1}^N \Na_i( x_j) . \qedhere
  \]
\end{proof}

The proof of the next lemma depends on the explicit formula for the projection $p_m$ from \eqref{eq:mproj} together with the second identity in Lemma \ref{l:xpasty} (for $s = r = \ell$). The proof is left out since it is almost identical to the proof of \cite[Lemma 5.4]{AgKa:PSM}. Notice that $p_{-1} := 0$ by convention.

\begin{lemma}\label{l:paproj}
Let $m \in \nn_0$ and $i \in \inn{\ell}$. We have the identity
\[
p_m \cd \Na_i(y_\ell) = \Na_i(y_\ell) \cd p_{m -1} .
\]
\end{lemma}

We are now prepared to state the main result in this section. In the proof we focus on the minor details which are not already covered by the argument given in the proof of \cite[Lemma 5.3]{AgKa:PSM}. Notice that, in the formulae here below, we are suppressing the representation of $C(\B CP_q^\ell)$ on the Hilbert space $L^2(\Om_M^0,h)$, see the discussion after Definition \ref{d:antihol}.

\begin{prop}\label{p:paproj}
Let $m \in \nn_0$. It holds that $p_m$ belongs to the domain of the closure $\ov{\Na}$ and that
\[
\ov{\Na}(p_m) =  \frac{1}{q^{2m}(1 - q^2)}\cd \Na(y_\ell) \big( p_m - q^2 p_{m-1} \big) 
\]
\end{prop}
\begin{proof}
  Using the commutator relation from Lemma \ref{l:xpasty} it can be verified that
  \[
\Na_i(y_\ell^n) = \sum_{k = 0}^{n-1} q^{2k} \Na_i(y_\ell) y_\ell^{n-1}  \quad \T{for all } n \in \nn .
\]
Using the formula for $\Na$ on elements from the coordinate algebra given in \eqref{eq:nabcoord} we then see that
\begin{equation}\label{eq:nabpow}
\Na(y_\ell^n) = \frac{1 - q^{2n}}{1 - q^2} \Na(y_\ell) y_\ell^{n-1}
\end{equation}
and hence that $\Na(y_\ell^n)$ converges in operator norm to $\frac{1}{1 - q^2} \Na(y_\ell) p_0$. This proves the relevant identity for $m = 0$. The general case where $m \in \nn$ now follows by an induction argument relying on the formula from \eqref{eq:mproj} and the identity in \eqref{eq:nabpow}. For more details we refer to the proof of \cite[Lemma 5.3]{AgKa:PSM}. 
\end{proof}

\subsection{Difference quotients}
In this subsection we are interested in computing the gradient $\Na(f)$ of an arbitrary $q$-Lipschitz function $f \in \Lip_{D_q}(I_q)$. Letting $I_q^\ci := \{ q^{2m} \mid m \in \nn_0\}$ denote the interior of the quantized interval, we start out by introducing the \emph{difference quotients} $D(f)$ and $E(f)$ in $C(I_q^\ci)$ associated to a continuous function $f \in C(I_q)$. These are defined by putting
\[
D(f)(x) := \frac{f(x) - f(x q^2)}{x(1 - q^2)} \, \, \T{ and } \, \, \,
E(f)(x) := D(f)(x q^{-2}) 
\]
for all $x \in I_q^\ci$ (using the convention that $f(q^{-2}) := 0$).


For each $m \in \nn_0$ we then record the formulae 
\begin{equation}\label{eq:diffproje}
\begin{split}
D(p_m) & = \frac{1}{q^{2m}(1 - q^2)} ( p_m  - q^2 p_{m-1}) \quad \T{ and } \\
E(p_m) & = \frac{1}{q^{2m}(1-q^2)} (p_{m + 1}  - q^2 p_m) ,
\end{split}
\end{equation}
recalling again that $p_{-1} := 0$. 

We may apply the difference quotients to obtain a convenient expression for the closed derivation $\ov{\Na}$ applied to functions in the linear span of the projections $p_m$ for $m \in \nn_0$.  In the formulae here below we are suppressing the representations $\rho^0$ and $\rho^1$ of $C(\B CP_q^\ell)$ on $L^2(\Om_M^0,h)$ and $L^2(\Om_M^1,h)$ (these representations are defined near \eqref{eq:repqua}).

\begin{prop}\label{p:delfin}
Let $f \in \T{span}\{ p_m \mid m \in \nn_0 \}$. We have the identities
\[
  \ov{\Na}(f) = \Na(y_\ell) D(f) = E(f) \Na(y_\ell) .
  \]
\end{prop}
\begin{proof}
  By linearity of the involved expressions we may fix $m \in \nn_0$ and focus on the case $f = p_m$. In this case, we obtain the relevant formulae from Proposition \ref{p:paproj} and Lemma \ref{l:paproj} together with \eqref{eq:diffproje}. 
\end{proof}

We shall now see how we can extract information about the gradient of an arbitrary $q$-Lipschitz function. 

\begin{prop}\label{p:delgen}
Let $f \in \Lip_{D_q}(I_q)$ and let $m \in \nn_0$. We have the identity
\[
\Na(f) \cd p_m  = \Na(y_\ell) D(f)(q^{2m}) \cd p_m .
\]
\end{prop}
\begin{proof}
  Define the function $g := f \cd ( p_m + p_{m + 1} )$ and notice that $g \in \T{span}\{ p_k \mid k \in \nn_0 \}$ and furthermore that 
\[
f(q^{2m}) = g(q^{2m}) \, \, , \, \, \, D(f)(q^{2m}) = D(g)(q^{2m}) \, \, \T{ and } \, \, \, f \cd E(p_m) = g \cd E(p_m) .  
\]
Using that $\ov{\Na}$ is a derivation in combination with Proposition \ref{p:delfin} we then get that
\[
\begin{split}
  \Na(f) \cd p_m & = \Na( f \cd p_m) - f \cd \ov{\Na}(p_m) = \Na(g \cd p_m) - f \cd E(p_m) \Na(y_\ell) \\
  & = \Na(g \cd p_m) - g \cd \ov{\Na}(p_m) = \Na(g) \cd p_m = \Na(y_\ell) D(g)(q^{2m}) \cd p_m \\
  & = \Na(y_\ell) D(f)(q^{2m}) \cd p_m . \qedhere
\end{split}
  \]
\end{proof}

\subsection{Estimates on Lipschitz functions}
In order to show that the pair $\big( \Lip_{D_q}(I_q), L_{D_q}^{\max} \big)$ is a compact quantum metric space, our strategy is to apply the slip-norm $L_{D_q}^{\max}$ to obtain estimates on an arbitrary $q$-Lipschitz function $f$. In fact, we need to control how fast the corresponding sequence $\big\{ f(q^{2m}) \}_{m = 0}^\infty$ converges to $f(0)$. Since the quantity $L_{D_q}^{\max}(f)$ dominates the operator norm of the gradient $\Na(f)$ applied to an arbitrary $q$-Lipschitz function it is relevant to relate the bounded operator $\Na(f)^* \Na(f)$ to the values of the function $f$. The next proposition in combination with Proposition \ref{p:delgen} provide us with a formula for $\Na(f)^* \Na(f)$ in terms of the difference quotient $D(f)$ and a single explicit positive continuous function on the quantized interval $I_q$.

\begin{prop}\label{p:Gmat}
It holds that
\[
\Na(y_\ell)^* \Na(y_\ell) =  q^{-1} y_\ell (1 - q^2 y_\ell) . 
\]
\end{prop}
\begin{proof}
We first recall that the identities in \eqref{eq:defsphere} imply that $y_\ell = q^{-2}(1- z_N^* z_N )$ and $z_N^* (1 - z_N z_N^*) z_N = q^2 y_\ell (1 - q^2 y_\ell)$. For every $i \in \inn{\ell}$ we then obtain from the considerations in the beginning of Subsection \ref{ss:derind} that
\[
\Na_i(y_\ell) = -q^{-2} \Na_i(z_N^* z_N) = - q^{-2} \Na_i(z_N^*) d_{K_\ell^{-1}}(z_N) = - q^{-3/2} (-q)^{i - N} u_{iN}^* z_N .
\]
Combining these observations with the identities from \eqref{eq:unitrans} and \eqref{eq:nabcoord} we obtain the desired result:
\[
\begin{split}
\Na(y_\ell)^* \Na( y_\ell) & = q^{-3} \cd \sum_{i = 1}^\ell q^{2(i - N)} z_N^* u_{iN} u_{iN}^* z_N \\
& = q^{-3} z_N^* \cd \big( \sum_{i = 1}^N q^{2(i-N)} u_{iN} u_{iN}^* \big) \cd z_N
- q^{-3} z_N^* z_N z_N^* z_N \\
& = q^{-3} z_N^* (1 - z_N z_N^*) z_N  = q^{-1} y_\ell (1 - q^2 y_\ell) . \qedhere
\end{split}
\]
\end{proof}

We apply the notation $G := q^{-1} y_\ell (1 - q^2 y_\ell)$ for the positive continuous function on the quantized interval $I_q$ appearing in Proposition \ref{p:Gmat}.

\begin{prop}\label{p:lipestim}
  Let $f \in \Lip_{D_q}(I_q)$ and let $x,y \in I_q$. We have the estimate
  \[
\big| f(x) - f(y) \big| \leq \frac{\sqrt{1 + q}}{\sqrt{1 - q}} \cd \big| \sqrt{x} - \sqrt{y} \big| \cd L_{D_q}^{\max}(f) .
\]
\end{prop}
\begin{proof}
  Let $m \in \nn_0$ and notice first of all that it suffices to establish the result for $x = q^{2m}$ and $y = q^{2(m+1)}$. Using Proposition \ref{p:delgen} and Proposition \ref{p:Gmat} together with the fact that $G(q^{2m}) \geq q^{2m} (1 - q^2)$ we arrive at the estimates:
    \[
    \begin{split}
      L_{D_q}^{\max}(f) & \geq \| \Na(f) \| \geq \| \Na(f) p_m \| = \big\| \Na(y_\ell) D(f)(q^{2m}) p_m \big\| \\
      & = G(q^{2m})^{1/2} \cd \big| D(f)(q^{2m}) \big|
      \geq \frac{ \big| f(q^{2m}) - f(q^{2(m+1)}) \big|}{q^m \sqrt{1 - q^2}}  \\
      & = \big| f(q^{2m}) - f(q^{2(m+1)}) \big| \cd \frac{\sqrt{1-q}}{(q^m - q^{m+1})\sqrt{1 + q}}. \qedhere
    \end{split}
  \]
\end{proof}

\subsection{Main theorems}

\begin{theorem}\label{thm:qms}
The pair $\big( \Lip_{D_q}(I_q),  L_{D_q}^{\max} \big)$ is a compact quantum metric space.
\end{theorem}
\begin{proof}
To ease the notation, put $C_q := \frac{\sqrt{1 + q}}{\sqrt{1 - q}}$. The finite diameter condition is a consequence of Proposition \ref{p:lipestim}. Indeed, for every $f \in \Lip_{D_q}(I_q)$ we get that
\[
\big\| [f] \big\|_{C(I_q)/\B C} \leq \| f - f(0) \| \leq C_q \cd L_{D_q}^{\max}(f) .
\]
  
Let $\ep > 0$ be given and choose $N \in \nn_0$ such that $C_q \cd q^N \leq \ep$. Let $\T{id} : C(I_q) \to C(I_q)$ denote the identity operator and define the positive unital map $\Psi : C(I_q) \to C(I_q)$ by putting
\[
\Psi(f)(x) := \fork{ccc}{f(x) & \T{for} & x \geq q^{2N} \\ f(q^{2N}) & \T{for} & x \leq q^{2N}} .
\]
Since the quantized interval $I_q$ only contains finitely many points which are larger than or equal to $q^{2N}$ we get that $\Psi$ has finite dimensional image. For each $f \in \Lip_{D_q}(I_q)$ we moreover obtain from Proposition \ref{p:lipestim} that
\[
\begin{split}
  \| f - \Psi(f) \| & = \sup\big\{ \big| f(x) - f(q^{2N}) \big| \, \mid \, \, x \leq q^{2N} \big\} \\
 & \leq \sup\big\{ C_q ( q^N - \sqrt{x}  ) \cd L_{D_q}^{\max}(f) \, \mid \, \, x \leq q^{2N} \big\} \\
 & \leq \ep \cd L_{D_q}^{\max}(f) .
\end{split}
  \]
  This shows that $(\T{id}, \Psi)$ is an $\ep$-approximation of the pair $\big( \Lip_{D_q}(I_q),  L_{D_q}^{\max}\big)$.

  It therefore follows from Theorem \ref{t:approx} that $\big(\Lip_{D_q}(I_q),L_{D_q}^{\max}\big)$ is a compact quantum metric space.  
\end{proof}

\begin{theorem}\label{t:mainthm}
The unital spectral triple $\big( \Lip_{D_q}(\B CP_q^\ell), L^2(\Om_M,h), D_q \big)$ is a spectral metric space.
\end{theorem}
\begin{proof}
This follows immediately by combining Theorem \ref{t:extlem} and Theorem \ref{thm:qms}.
\end{proof}

As a corollary to our main theorem we recover the main result from \cite{AgKa:PSM} as a special case, see the discussion in Subsection \ref{ss:dabsit} and in particular Proposition \ref{p:dabsitdada}.

\begin{cor}
The Lipschitz algebra version of the D\k{a}browski-Sitarz spectral triple $\big( \Lip_{\dirac_q}(S_q^2), H_+ \op H_- , \dirac_q \big)$ is a spectral metric space.
\end{cor}


\begin{thebibliography}{10}

\bibitem{AgKa:PSM}
Konrad Aguilar and Jens Kaad.
\newblock The {P}odle\'{s} sphere as a spectral metric space.
\newblock {\em J. Geom. Phys.}, 133:260--278, 2018.

\bibitem{AgKaKy:PSC}
Konrad Aguilar, Jens Kaad, and David Kyed.
\newblock The {P}odle\'{s} spheres converge to the sphere.
\newblock {\em Comm. Math. Phys.}, 392(3):1029--1061, 2022.

\bibitem{BaSk:UMD}
Saad Baaj and Georges Skandalis.
\newblock Unitaires multiplicatifs et dualit\'{e} pour les produits crois\'{e}s
  de {$C^*$}-alg\`ebres.
\newblock {\em Ann. Sci. \'{E}cole Norm. Sup. (4)}, 26(4):425--488, 1993.

\bibitem{BMR:DSS}
Jean Bellissard, Matilde Marcolli, and Kamran Reihani.
\newblock Dynamical systems on spectral metric spaces.
\newblock {\em Preprint}, 2010.
\newblock \url{https://arxiv.org/abs/1008.4617}.

\bibitem{Bra:GFC}
Branimir \'Ca\'ci\'c.
\newblock Geometric foundations for classical $u(1)$-gauge theory on
  noncommutative manifolds.
\newblock {\em Preprint}, 2023.
\newblock \url{https://arxiv.org/abs/2301.01749}.


\bibitem{CAMW:SDM}
Eric Cagnache, Francesco D'Andrea, Pierre Martinetti, and Jean-Christophe Wallet.
  \newblock The spectral distance in the {M}oyal plane.
  \newblock {\em J. Geom. Phys.}, 61(10):1881--1897, 2011. 


\bibitem{CaPa:EST}
Partha~Sarathi Chakraborty and Arupkumar Pal.
\newblock Equivariant spectral triples on the quantum {${\rm SU}(2)$} group.
\newblock {\em $K$-Theory}, 28(2):107--126, 2003.

\bibitem{Chr:WDO}
Erik Christensen.
\newblock On weakly {$D$}-differentiable operators.
\newblock {\em Expo. Math.}, 34(1):27--42, 2016.

\bibitem{Con:CFH}
Alain Connes.
\newblock Compact metric spaces, {F}redholm modules, and hyperfiniteness.
\newblock {\em Ergodic Theory Dynam. Systems}, 9(2):207--220, 1989.

\bibitem{Con:NCG}
Alain Connes.
\newblock {\em Noncommutative geometry}.
\newblock Academic Press, Inc., San Diego, CA, 1994.

\bibitem{CoMo:TIII}
Alain Connes and Henri Moscovici.
\newblock Type {III} and spectral triples.
\newblock In {\em Traces in number theory, geometry and quantum fields},
  Aspects Math., E38, pages 57--71. Friedr. Vieweg, Wiesbaden, 2008.

\bibitem{CoSu:STN}
Alain Connes and Walter~D. van Suijlekom.
\newblock Spectral truncations in noncommutative geometry and operator systems.
\newblock {\em Comm. Math. Phys.}, 383(3):2021--2067, 2021.

\bibitem{CoSu:TRO}
Alain Connes and Walter~D. van Suijlekom.
\newblock Tolerance relations and operator systems.
\newblock {\em Acta Sci. Math. (Szeged)}, 88(1-2):101--129, 2022.

\bibitem{DaDa:DQP}
Francesco D'Andrea and Ludwik D\k{a}browski.
\newblock Dirac operators on quantum projective spaces.
\newblock {\em Comm. Math. Phys.}, 295(3):731--790, 2010.

\bibitem{DaDaLa:NGQPP}
Francesco D'Andrea, Ludwik D\k{a}browski, and Giovanni Landi.
\newblock The noncommutative geometry of the quantum projective plane.
\newblock {\em Rev. Math. Phys.}, 20(8):979--1006, 2008.

\bibitem{DaBuSo:DDQ}
Biswarup Das, R\'{e}amonn \'{O}~Buachalla, and Petr Somberg.
\newblock A {D}olbeault-{D}irac spectral triple for quantum projective space.
\newblock {\em Doc. Math.}, 25:1079--1157, 2020.

\bibitem{DLSSV:DOS}
Ludwik D\k{a}browski, Giovanni Landi, Andrzej Sitarz, Walter van Suijlekom, and
  Joseph~C. V\'{a}rilly.
\newblock The {D}irac operator on {${\rm SU}_q(2)$}.
\newblock {\em Comm. Math. Phys.}, 259(3):729--759, 2005.

\bibitem{DaSi:DSP}
Ludwik D\k{a}browski and Andrzej Sitarz.
\newblock Dirac operator on the standard {P}odle\'{s} quantum sphere.
\newblock In {\em Noncommutative geometry and quantum groups ({W}arsaw, 2001)},
  volume~61 of {\em Banach Center Publ.}, pages 49--58. Polish Acad. Sci. Inst.
  Math., Warsaw, 2003.

\bibitem{Dri:QG}
Vladimir Drinfeld.
\newblock Quantum groups.
\newblock {\em Zap. Nauchn. Sem. Leningrad. Otdel. Mat. Inst. Steklov. (LOMI)},
  155(Differentsial\cprime naya Geometriya, Gruppy Li i Mekh. VIII):18--49,
  193, 1986.

\bibitem{GKK:QI}
Thomas Gotfredsen, Jens Kaad, and David Kyed.
\newblock Gromov-{H}ausdorff convergence of quantised intervals.
\newblock {\em J. Math. Anal. Appl.}, 500(2):Paper No. 125131, 13, 2021.

\bibitem{HeKo:LDI}
Istv\'{a}n Heckenberger and Stefan Kolb.
\newblock The locally finite part of the dual coalgebra of quantized
  irreducible flag manifolds.
\newblock {\em Proc. London Math. Soc. (3)}, 89(2):457--484, 2004.

\bibitem{HeKo:RQF}
Istv\'{a}n Heckenberger and Stefan Kolb.
\newblock De {R}ham complex for quantized irreducible flag manifolds.
\newblock {\em J. Algebra}, 305(2):704--741, 2006.

\bibitem{HoSz:QSPS}
Jeong~Hee Hong and Wojciech Szyma\'nski.
\newblock Quantum spheres and projective spaces as graph algebras.
\newblock {\em Comm. Math. Phys.}, 232(1):157--188, 2002.

\bibitem{IzNeTu:PBD}
Masaki Izumi, Sergey Neshveyev, and Lars Tuset.
\newblock Poisson boundary of the dual of {${\rm SU}_q(n)$}.
\newblock {\em Comm. Math. Phys.}, 262(2):505--531, 2006.

\bibitem{Jim:QYB}
Michio Jimbo.
\newblock A {$q$}-difference analogue of {$U({\germ g})$} and the
  {Y}ang-{B}axter equation.
\newblock {\em Lett. Math. Phys.}, 10(1):63--69, 1985.

\bibitem{Kaa:UKM}
Jens Kaad.
\newblock The unbounded {K}asparov product by a differentiable module.
\newblock {\em J. Noncommut. Geom.}, 15(2):423--487, 2021.

\bibitem{Kaa:ExPr}
Jens Kaad.
\newblock External products of spectral metric spaces.
\newblock {\em To appear in Pure and Applied Functional Analysis}, pages 1--36,
  2023.
\newblock \url{https://arxiv.org/abs/2304.03979}.

\bibitem{KaKy:SU2}
Jens Kaad and David Kyed.
\newblock The quantum metric structure of quantum {$SU(2)$}.
\newblock {\em Preprint}, 2022.
\newblock \url{https://arxiv.org/abs/2205.06043}.

\bibitem{KaLe:SFU}
Jens Kaad and Matthias Lesch.
\newblock Spectral flow and the unbounded {K}asparov product.
\newblock {\em Adv. Math.}, 248:495--530, 2013.

\bibitem{KaSe:TST}
Jens Kaad and Roger Senior.
\newblock A twisted spectral triple for quantum {$SU(2)$}.
\newblock {\em J. Geom. Phys.}, 62(4):731--739, 2012.

\bibitem{KlSc:QGR}
Anatoli Klimyk and Konrad Schm\"{u}dgen.
\newblock {\em Quantum groups and their representations}.
\newblock Texts and Monographs in Physics. Springer-Verlag, Berlin, 1997.

\bibitem{Kra:DOQ}
Ulrich Kr\"{a}hmer.
\newblock Dirac operators on quantum flag manifolds.
\newblock {\em Lett. Math. Phys.}, 67(1):49--59, 2004.

\bibitem{KRS:RFH}
Ulrich Kr\"{a}hmer, Adam Rennie, and Roger Senior.
\newblock A residue formula for the fundamental {H}ochschild 3-cocycle for
  {$SU_q(2)$}.
\newblock {\em J. Lie Theory}, 22(2):557--585, 2012.

\bibitem{KrTu:DDQ}
Ulrich Kr\"{a}hmer and Matthew Tucker-Simmons.
\newblock On the {D}olbeault-{D}irac operator of quantized symmetric spaces.
\newblock {\em Trans. London Math. Soc.}, 2(1):33--56, 2015.

\bibitem{KyNe:FMS}
David Kyed and Ryszard Nest.
\newblock Finiteness of metrics on state spaces.
\newblock {\em To appear in Bull. London Math. Soc.}, pages 1--8, 2023.

\bibitem{Lat:QGH}
Fr\'{e}d\'{e}ric Latr\'{e}moli\`ere.
\newblock The quantum {G}romov-{H}ausdorff propinquity.
\newblock {\em Trans. Amer. Math. Soc.}, 368(1):365--411, 2016.

\bibitem{Lat:MGH}
Fr\'{e}d\'{e}ric Latr\'{e}moli\`ere.
\newblock The modular {G}romov-{H}ausdorff propinquity.
\newblock {\em Dissertationes Math.}, 544:70, 2019.

\bibitem{Lat:GHM}
Fr\'{e}d\'{e}ric Latr\'{e}moli\`ere.
\newblock The {G}romov-{H}ausdorff propinquity for metric spectral triples.
\newblock {\em Adv. Math.}, 404(part A):Paper No. 108393, 56, 2022.

\bibitem{Mat:DDQ}
Marco Matassa.
\newblock On the {D}olbeault-{D}irac operators on quantum projective spaces.
\newblock {\em J. Lie Theory}, 28(1):211--244, 2018.

\bibitem{Mes:UCN}
Bram Mesland.
\newblock Unbounded bivariant {$K$}-theory and correspondences in
  noncommutative geometry.
\newblock {\em J. Reine Angew. Math.}, 691:101--172, 2014.

\bibitem{MeRe:NMU}
Bram Mesland and Adam Rennie.
\newblock Nonunital spectral triples and metric completeness in unbounded
  {$KK$}-theory.
\newblock {\em J. Funct. Anal.}, 271(9):2460--2538, 2016.

\bibitem{MiKa:HSV}
Max~Holst Mikkelsen and Jens Kaad.
\newblock The {H}aar state on the {V}aksman-{S}oibelman quantum spheres.
\newblock {\em Mathematica Scandinavica}, 129(2):337--357, 2023.

\bibitem{Nag:HQG}
Gabriel Nagy.
\newblock On the {H}aar measure of the quantum {${\rm SU}(N)$} group.
\newblock {\em Comm. Math. Phys.}, 153(2):217--228, 1993.

\bibitem{NeTu:LFQ}
Sergey Neshveyev and Lars Tuset.
\newblock A local index formula for the quantum sphere.
\newblock {\em Comm. Math. Phys.}, 254(2):323--341, 2005.

\bibitem{NeTu:DOC}
Sergey Neshveyev and Lars Tuset.
\newblock The {D}irac operator on compact quantum groups.
\newblock {\em J. Reine Angew. Math.}, 641:1--20, 2010.

\bibitem{RaRu:MTP}
Svetlozar~T. Rachev and Ludger R\"{u}schendorf.
\newblock {\em Mass transportation problems. {V}ol. {I}}.
\newblock Probability and its Applications (New York). Springer-Verlag, New
  York, 1998.
\newblock Theory.

\bibitem{Rie:MSA}
Marc~A. Rieffel.
\newblock Metrics on states from actions of compact groups.
\newblock {\em Doc. Math.}, 3:215--229, 1998.

\bibitem{Rie:MSS}
Marc~A. Rieffel.
\newblock Metrics on state spaces.
\newblock {\em Doc. Math.}, 4:559--600, 1999.

\bibitem{Rie:GHD}
Marc~A. Rieffel.
\newblock Gromov-{H}ausdorff distance for quantum metric spaces.
\newblock {\em Mem. Amer. Math. Soc.}, 168(796):1--65, 2004.
\newblock Appendix 1 by Hanfeng Li, Gromov-Hausdorff distance for quantum
  metric spaces. Matrix algebras converge to the sphere for quantum
  Gromov-Hausdorff distance.

\bibitem{Rie:MBM}
Marc~A. Rieffel.
\newblock Matricial bridges for ``matrix algebras converge to the sphere''.
\newblock In {\em Operator algebras and their applications}, volume 671 of {\em
  Contemp. Math.}, pages 209--233. Amer. Math. Soc., Providence, RI, 2016.

\bibitem{Rie:CFT}
Marc~A. Rieffel.
\newblock Convergence of {F}ourier truncations for compact quantum groups and
  finitely generated groups.
\newblock {\em J. Geom. Phys.}, 192:Paper No. 104921, 13, 2023.

\bibitem{VaSo:AQS}
Yan Soibelman and Leonid Vaksman.
\newblock Algebra of functions on the quantum group {${\rm SU}(n+1),$} and
  odd-dimensional quantum spheres.
\newblock {\em Algebra i Analiz}, 2(5):101--120, 1990.

\bibitem{Wor:CMP}
Stanis\l{}aw~L. Woronowicz.
\newblock Compact matrix pseudogroups.
\newblock {\em Comm. Math. Phys.}, 111(4):613--665, 1987.

\bibitem{Wor:CQG}
Stanis\l{}aw~L. Woronowicz.
\newblock Compact quantum groups.
\newblock In {\em Sym\'{e}tries quantiques ({L}es {H}ouches, 1995)}, pages
  845--884. North-Holland, Amsterdam, 1998.
  
  
\end{thebibliography}

\bibliographystyle{plain}

\end{document}